\documentclass{amsart}

\usepackage{amssymb} \usepackage{amsfonts} \usepackage{amsmath}
\usepackage{amsthm} \usepackage{epsfig} 
\usepackage{color}
\usepackage{pinlabel}
\usepackage{amscd}
\usepackage[all]{xy}
\usepackage{pinlabel}
\usepackage{tcolorbox}

\newtheorem{lemma}{Lemma}%[section]
\newtheorem{teo}[lemma]{Theorem}
\newtheorem{prop}[lemma]{Proposition} 
\newtheorem{cor}[lemma]{Corollary}

\theoremstyle{definition}

\newtheorem{quest}[lemma]{Question}

\newtheorem{example}[lemma]{Example}

\theoremstyle{remark}
\newtheorem{rem}[lemma]{Remark} 

\newcommand{\matr} [4] {\big({\tiny\begin{array}{@{}c@{\ }c@{}} #1 & #2 \\ #3 & #4 \\ \end{array}} \big)}

\newcommand{\Iso}{{\rm Isom}}

\newcommand{\matX}{\ensuremath {\mathbb{X}}}

\newcommand{\matR} {\ensuremath {\mathbb{R}}}

\newcommand{\matZ} {\ensuremath {\mathbb{Z}}}

\newcommand{\matH} {\ensuremath {\mathbb{H}}}
\newcommand{\matS} {\ensuremath {\mathbb{S}}}

\newcommand{\calC} {\ensuremath {\mathcal{C}}}

\newcommand{\SO} {\ensuremath {{\rm SO}}}

\newcommand{\rk} {\ensuremath {{\rm rk}}}
\newcommand{\Vol} {\ensuremath {{\rm Vol}}}

\newcommand{\Isom} {\ensuremath {{\rm Isom}}}
\newcommand{\lk} {\ensuremath {{\rm link}}}

\usepackage{mathtools}

\newcommand{\nota} [1] {\caption{\footnotesize{#1}}}

\author{Ludovico Battista}
\address{}
\email{}

\author{Bruno Martelli}
\address{}
\email{}

\title{Hyperbolic 4-manifolds 
with perfect circle-valued Morse functions}

\begin{document}

\begin{abstract}
We exhibit some (compact and cusped) finite-volume hyperbolic four-manifolds $M$ with perfect circle-valued Morse functions, that is circle-valued Morse functions $f\colon M \to S^1$ with only index 2 critical points. We construct in particular one example where every generic circle-valued function is homotopic to a perfect one. 

An immediate consequence is the existence of infinitely many finite-volume (compact and cusped) hyperbolic 4-manifolds $M$ having a handle decomposition with bounded numbers of 1- and 3-handles, so with bounded Betti numbers $b_1(M), b_3(M)$ and rank $\rk(\pi_1(M))$.
\end{abstract}

\maketitle

\section*{Introduction} \label{introduction:section}
One of the most intriguing phenomena in 3-dimensional topology is the existence of many finite-volume hyperbolic 3-manifolds $M$ that fiber over $S^1$. On such manifolds, the surface fiber $\Sigma$ generates a normal subgroup $\pi_1(\Sigma) \triangleleft \pi_1(M)$ and determines a geometrically infinite covering $\widetilde M \to M$ diffeomorphic to $\Sigma \times \matR$. 
Every finite-volume hyperbolic 3-manifold is finitely covered by a 3-manifold that fibers over $S^1$ by a celebrated theorem of Agol and Wise \cite{A, W}. All the hyperbolic manifolds and orbifolds in this paper are tacitly assumed to be complete.

We could ask whether such fibrations occur also in higher dimension $n\geq 4$. The answer is certainly negative in even dimensions, for a simple reason: by the generalised Gauss -- Bonnet formula, the Euler characteristic of an even-dimensional finite-volume hyperbolic manifold does not vanish, while that of a fibration does.

The following seems a more sensible question to ask. 

\begin{quest}
Are there finite-volume hyperbolic $n$-manifolds with perfect circle-valued Morse functions in all dimensions $n$?
\end{quest}

We explain the terminology. A \emph{circle-valued Morse function} on a compact manifold $M$, possibly with boundary, is a smooth map $f \colon M\to S^1$ such that $f|_{\partial M}$ has no critical points and $f$
has finitely many critical points, all of non degenerate type. This implies in particular that $f|_{\partial M}$ is a fibration. 

Every finite-volume hyperbolic $n$-manifold $M$ is either closed or the interior of a compact manifold with boundary $M^*$. In the latter case, a circle-valued Morse function on $M$ is by definition the restriction of one on $M^*$. 

As it is customary in Morse theory, a circle-valued Morse function $f$ on a compact manifold $M$ yields a \emph{circular handle decomposition} for $M$, obtained by starting from any regular fiber, thickening it, adding a $i$-handle corresponding to each index-$i$ critical point, and then finally gluing the left and right vertical boundaries of the resulting manifold via some diffeomorphism.

A circle-valued Morse function $f\colon M \to S^1$ is \emph{perfect} if it has exactly $|\chi(M)|$ critical points. In general we have
$$\chi(M) = \sum (-1)^ic_i$$ 
where $c_i$ is the number of critical points of index $i$. Therefore $f$ is perfect if and only if it has the minimum possible number of critical points allowed by $\chi(M)$, and this holds precisely when all the indices of the critical points have the same parity.

In odd dimensions we have $\chi(M)=0$ and hence a perfect circle-valued Morse function is just a fibration. In dimension 2 it is a simple exercise to prove that every closed orientable surface has a perfect circle-valued Morse function. 

In dimension 4 a finite-volume hyperbolic four-manifold $M$ has $\chi(M)>0$ and a circle-valued Morse function on $M$ is perfect if and only if the critical points have index 0, 2, or 4. Since $M$ is connected and not simply connected, we deduce easily (by looking at the corresponding circular handle decomposition) that in fact only the index 2 is allowed.
Our main contribution is to provide some first examples, both in the cusped and in the compact setting.

\begin{teo} \label{main:teo}
There are some finite-volume (compact and cusped) hyperbolic 4-manifolds with perfect circle-valued Morse functions.
\end{teo}

\begin{table}
\begin{center}
\begin{tabular}{cccccccc}
Name & cusps & Euler & $b_0$ & $b_1$ & $b_2$ & $b_3$ & $b_4$ \\
\hline
$W$ & 5 & 2 & 1 & 5 & 10 & 4 & 0 \\
$X$ & 24 & 8 & 1 & 21 & 51 & 23 & 0 \\
$Y$ & 1 & 1/12 & & & & & \\
$Z$ & 0 & 272 & 1 & 115 & 500 & 115 & 1
\end{tabular}
\vspace{.2 cm}
\nota{The hyperbolic 4-manifolds $W, X, Z$ and 4-orbifold $Y$ that we consider here. Each has some perfect circle-valued Morse functions. For each manifold or orbifold we display the number of cusps, the Euler characteristic, and the Betti numbers over $\matR$ (only for manifolds).}
\label{WXYZ:table}
\end{center}
\end{table}

We construct in particular three examples: two cusped hyperbolic 4-manifolds $W, X$ and one closed hyperbolic 4-manifold $Z$, each equipped with some perfect circle-valued Morse function. We also build an additional hyperbolic 4-orbifold $Y$ that has a perfect circle-valued Morse function (in some natural sense) with particularly small fibers. Some basic information on the topology of these four objects is collected in Table \ref{WXYZ:table}. 

This produces an immediate corollary. If $M$ is a cusped hyperbolic manifold, a handle decomposition for $M$ is by definition one for the compactification $M^*$.

\begin{cor}
There are infinitely many finite-volume (compact and cusped) hyperbolic 4-manifolds $M$  with a handle decomposition with bounded numbers of 1- and 3-handles, hence with bounded Betti numbers $b_1(M)$ and $b_3(M)$ and rank of $\pi_1(M)$. 
\end{cor}
\begin{proof}
If $M$ is cusped, in what follows we work with the compactification $M^*$, still denoted by $M$ for simplicity.

If $M$ has a perfect circle-valued Morse function $f$, it inherits a circular handle decomposition, which further decomposes into a handle decomposition as follows: fix a handle decomposition of a regular fiber $N$ of $f$, thicken it, add a 2-handle for each singular point of $f$, and then add a $(i+1)$-handle for each $i$-handle of $N$ to close everything up.

For every $n \geq 2$ we can construct a cyclic covering $M_n \to M$ and a lift $f_n \colon M_n \to S^1$ by unwrapping $n$ times along $f$. The lifted $f_n$ is a circle-valued Morse function with the same regular fiber $N$ as above, so $M_n$ has a handle decomposition with a fixed number of $i$-handles for all $i\neq 2$.
\end{proof}

Recall that there are only finitely many hyperbolic 4-manifolds $M$ with bounded $b_2(M)$, since $\Vol(M) = \frac{4\pi^2}3 \chi(M)$ and $\chi(M) \leq 2 + b_2(M)$, and there are only finitely many finite-volume hyperbolic 4-manifolds $M$ with bounded volume \cite{Wang}.

The main purpose of this work is to show that perfect circle-valued Morse functions are quite frequent, at least among the very limited types of manifolds that we are able to investigate at present, that is those that decompose into right-angled polytopes. We construct the manifolds $W,X,Z$ by colouring some well-known right-angled polytopes, and then we build some circle-valued Morse functions by combining the techniques of Bestvina -- Brady \cite{BB} and Jankiewicz -- Norin -- Wise \cite{JNW}, plus some additional arguments that are presented in this paper.

For the manifold $W$ listed in Table \ref{WXYZ:table} we are able to determine precisely the integral cohomology classes that are represented by a perfect circle-valued Morse function.

\begin{teo} \label{W:teo}
The cusped hyperbolic 4-manifold $W$ has $H^1(W; \matR) = \matR^5$. The cohomology classes that are represented by some perfect circle-valued Morse function form the intersection
$$U \cap H^1(W; \matZ)$$
where $U= \{x_i \neq 0\}$ is the complement of the coordinate hyperplanes.
\end{teo}

The open set $U$ is dense, so a homologically generic map $W \to S^1$ is homotopic to a perfect Morse function. Moreover, the set $U$ is \emph{polytopal}, that is it is the cone over some open facets of a polytope. The polytope here is the hyperoctahedron, convex hull of the vertices $\pm e_i$, with $2^5 = 32$ facets. So $U$ has 32 connected components.

Of course we know from Thurston \cite{Th} that a similar theorem holds for any hyperbolic 3-manifold, with $U$ being the cone over some faces of the unit ball polytope of the Thurston norm. (Recall that a perfect circle-valued Morse function in dimension 3 is a fibration.) 

Here is another immediate corollary of Theorem \ref{main:teo}.

\begin{cor}
There are some geometrically infinite hyperbolic 4-manifolds $M$ that are diffeomorphic to $N \times [0,1]$ with infinitely many 2-handles attached on both sides, where $N$ is (the interior of) a 3-manifold that is either closed or bounded by tori. Here $\pi_1(M)$ is finitely generated but not finitely presented and $b_2(M) = \infty$.
\end{cor}
\begin{proof}
Pick a perfect circle-valued Morse function $f\colon M_0 \to S^1$ on a closed or cusped hyperbolic 4-manifold $M_0$. It lifts to a Morse function $\tilde f \colon \widetilde M_0 \to \matR$ on the abelian covering $\widetilde M_0$ determined by $\ker f$. The manifold $M = \widetilde M_0$ is as stated. 

We have $b_2(M) = \infty$ because there are infinitely many 2-handles. If $N$ is a regular fiber, any finite set of generators for $\pi_1(N)$ generates also $\pi_1(M)$, which is hence finitely generated. If $\pi_1(M)$ were also finitely presented, then $H_2(\pi_1(M))$ would be finitely generated, but the asphericity of $M$ implies that $H_2(M) = H_2(\pi_1(M))$, a contradiction.
\end{proof}

In all the cases we could investigate here, we found strong numerical evidence that such a geometrically infinite $M$ should be infinitesimally and hence locally rigid. In one case we have a rigorous computer-assisted proof, see Theorem \ref{rigid:teo}. 
The theoretical and computational aspects of these results will be exposed in a forthcoming paper \cite{Bat}.

Many results exposed in this paper were obtained by using both Regina \cite{BBP} and SnapPy \cite{Sna} inside a Sage environment. 

\subsection*{Summary of the paper}
The hyperbolic 4-manifolds $W$, $X$, and $Z$ are all constructed by colouring appropriately the facets of three well-known right-angled polytopes, that are respectively $P_4$, the ideal 24-cell, and the compact right-angled 120-cell. We recall in Section \ref{colours:subsection} how a coloured right-angled polytope $P$ defines a manifold $M$ tessellated into copies of $P$. In Section \ref{cube:subsection} we build a cube complex $C$ dual to this tessellation.

In Section \ref{states:subsection} we introduce the notion of \emph{real state}, that extends that of a state given in \cite{JNW}. A real state $s$ is simply a real number assigned to every facet of $P$. It produces a cohomology class $[s] \in H^1(M, \matR)$ and a piecewise-linear \emph{diagonal map} $f\colon \widetilde M \to \matR$ on the universal cover representing that class, as explained in Section \ref{diagonal:subsection}. We then use Bestvina -- Brady theory to analyse this map carefully in Section \ref{ad:subsection}. When $[s]$ is integral and certain conditions are fulfilled, the map $f$ descends to a map $f\colon M \to \matR/_\matZ = S^1$ that can be smoothened to a perfect circle-valued Morse function. The conditions are stated quite generally in Theorem \ref{ad:teo}.

In Section \ref{manifolds:section} we use this machinery to build $W, X, Y, Z$ and their perfect circle-valued Morse functions. In Section \ref{W:subsection} we construct $W$, prove Theorem \ref{W:teo}, and analyse the \emph{singular fibers} of $f$ in one particular case. In Section \ref{X:subsection} we construct $X$ and analyse the singular fibers of 63 distinct maps $f$ that arise naturally from the combinatorics. We get 63 distinct hyperbolic 3-manifolds with as much as 28 cusps each, all distinguished by their hyperbolic volume, listed in Tables \ref{fibers:table} and \ref{fibers2:table}. 
The infinitesimal rigidity of all the abelian covers considered is verified numerically in all these cases. 

Among these 63 maps $f$, one is particularly interesting due to its many symmetries. In Section \ref{Y:subsection} we use this map to build a very small orbifold $Y$ with an appropriate kind of circle-valued Morse function, for which we can fully determine both the \emph{singular} and \emph{regular} fibers: they both belong to the first segment of the census of cusped hyperbolic 3-manifolds \cite{CHW}. Finally, we produce the compact example $Z$ in Section \ref{Z:subsection}.

We make some comments and raise some open questions in Section \ref{questions:section}.

\subsection*{Acknowledgements} We thank Marc Culler and Matthias Goerner for helpful discussions on SnapPy.

\section{The construction} \label{colourings:section}

Our aim is now to construct some compact and cusped hyperbolic 4-manifolds equipped with perfect circle-valued Morse functions. The manifolds are obtained by assembling some right-angled polytopes, and then applying some techniques of Bestvina -- Brady \cite{BB} and Jankiewicz -- Norin -- Wise \cite{JNW}, plus some additional arguments.

\subsection{Colours} \label{colours:subsection}
All the hyperbolic 4-manifolds we consider here are built by colouring the facets of a right-angled polytope. This is a simple and fruitful technique already considered in other contexts, see for instance \cite{IMM, KM}.

We briefly recall how it works. Let $P\subset \matX^n$ be a right-angled finite polytope in some space $\matX^n = \matH^n, \matR^n$ or $\matS^n$. A \emph{$c$-colouring} of $P$ is the assignment of a colour (taken from some finite set of $c$ elements) to each facet of $P$, such that incident facets have distinct colours. We typically use $\{1,\ldots, c\}$ as a palette of colours and always suppose that every colour is painted on at least one facet.

A colouring defines a manifold $M$, obtained by mirroring $P$ iteratively along its coloured facets (in any order). More precisely, consider $2^c$ disjoint copies of $P$ denoted as $P^v$ where $v\in \matZ_2^c$ varies. We identify each facet $F$ of $P^v$ via the identity map to the same facet of $P^{v+e_j}$ where $j$ is the colour of $F$. 

The result of these identifications is a $n$-manifold $M$ having the same geometry $\matX^n$ of $P$, tessellated into $2^c$ copies of $P$. 

From a more algebraic point of view, the manifold $M$ can be constructed as follows. Let $\Gamma< \Iso(\matX^n)$ be the Coxeter group generated by the reflections along the facets of $P$. We define a homomorphism $\Gamma \to \matZ_2^c$ by sending the reflection along a facet $F$ to $e_i$, where $i$ is the colour assigned to $F$. The kernel $\Gamma'$ of the homomorphism is torsion-free and we get $M=\matX^n/\Gamma'$. We can generalise naturally this algebraic construction by assigning arbitrary vectors in $\matZ_2^c$ to the facets, requiring that vectors attached to any set of facets with non-empty intersection are independent, but we will not need this generalisation here: see \cite{KSla} for an introduction.

If we permute the colours in the palette the resulting manifold $M$ is unaffected, so colourings are typically considered only up to such permutations: in other words, we should think of a colouring as a partition of the facets of $P$.

We are interested in colourings with a small number of colours and many symmetries, since they produce manifolds that are reasonably small and with many isometries. Here are some examples:

\begin{figure}
 \begin{center}
  \includegraphics[width = 3.2 cm]{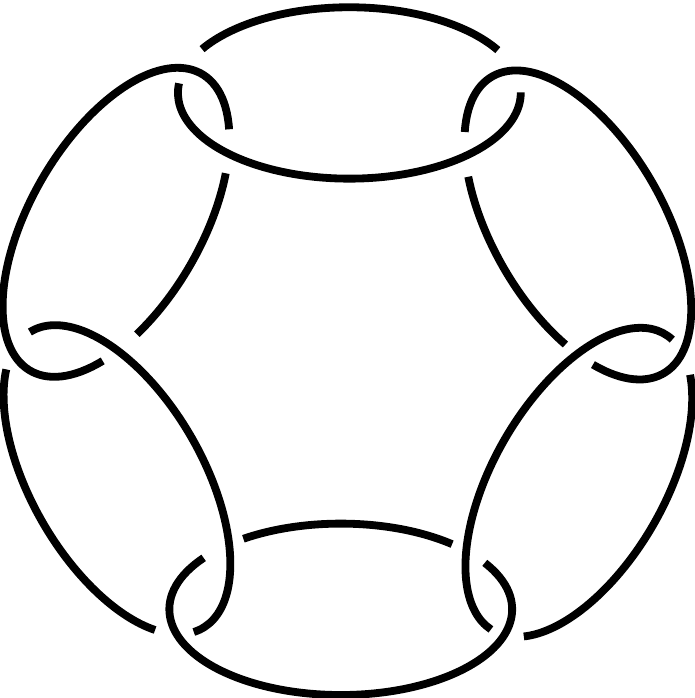}
 \end{center}
 \nota{The minimally twisted chain link with 6 components. }
 \label{chainlink:fig}
\end{figure}

\begin{itemize}
\item The $n$-cube has a unique $n$-colouring, where opposite facets are coloured with the same colour. This colouring produces a flat torus. More generally, one can prove that any colouring on the $n$-cube produces a flat torus.
\item The right-angled spherical $n$-simplex has a unique colouring, that produces the spherical manifold $S^n$.
\item The ideal octahedron in $\matH^3$ has a unique 2-colouring. The colouring produces a cusped hyperbolic 3-manifold which is the complement of the minimally twisted chain link with 6 components shown in Figure \ref{chainlink:fig}, see \cite{KM}.
\item The ideal 24-cell in $\matH^4$ has a unique 3-colouring. It produces a hyperbolic 4-manifold with 24 cusps with 3-torus sections, already considered in \cite{KM, MR}.
\item There is a sequence of hyperbolic right-angle polytopes $P_3,\ldots, P_8$ with $P_n \subset \matH^n$. Some natural symmetric colourings were constructed in \cite{IMM} for each $P_n$. These produce some hyperbolic $n$-manifolds.
\end{itemize}

The right-angled polytopes that we use here are $P_4$, the ideal 24-cell, and the right-angled 120-cell. We will describe their colourings later in detail.

\subsection{The dual cube complex} \label{cube:subsection} 
A coloured right-angled polytope $P$ produces a manifold $M$ tessellated into $2^c$ copies of $P$, denoted as $P^v$ with varying $v\in \matZ_2^c$. We now construct a cube complex $C$ dual to this tessellation. %Dual to this tessellation we have a finite cube complex $C\subset M$, that we think as topologically embedded in $M$ (we indicate with the letter $C$ both the cube complex and its support, for simplicity). 

We work in the PL category. We consider $P$ as a compact Euclidean polytope (using the Klein model for $\matH^n$). The PL manifold $M$ is decomposed into $2^c$ identical copies $P^v$ of $P$, possibly with some ideal vertices: these are the pre-images of the ideal vertices of $P$, and they should be removed to get $M$. We call this an (ideal) polyhedral decomposition for $M$. A \emph{$k$-face} of this decomposition is by definition a $k$-face of some $P^v$. By convention each $P^v$ is a $n$-face, but the ideal vertices are not 0-faces (while the finite vertices are 0-faces).

We fix a barycentric subdivision of $P$, which lifts to a barycentric subdivision of the polyhedral decomposition of $M$. We define $C$ as the dual decomposition to the polyhedral decomposition of $M$ inside its barycentric subdivision. More precisely, we do the following. Every $k$-face $F$ in the polyhedral decomposition has a barycenter $v$, and the dual $(n-k)$-cell in $C$ is defined as the union of all the simplexes of the barycentric subdivision that intersect $F$ in $v$. 

\begin{prop}
The construction produces a cube complex $C\subset M$ onto which $M$ retracts. There is a natural 1-1 correspondence between the $k$-faces of the polyhedral decomposition of $M$ and the $(n-k)$-cubes of $C$, which reverses all containments.
If there are no ideal vertices we have $C=M$. In general $M\setminus C$ consists of cusps.
\end{prop}
\begin{proof}
The barycentric subdivision of the polyhedral decomposition of $M$ can be constructed abstractly via the following standard two-steps procedure: (i) consider the poset of the faces of the polyhedral decomposition ordered by inclusion, and (ii) from the poset, define an abstract simplicial complex by considering every element as a vertex and every ascending chain as a simplex.

By construction the cell dual to a $k$-face $F$ is the simplicial subcomplex of the barycentric subdivision obtained from the chains in the poset formed by all the faces that contain $F$. Since $P$ is right-angled, we deduce easily that the poset is isomorphic to the poset of the faces of a $(n-k)$-cube. Therefore the cell is a barycentric subdivision of a $(n-k)$-cube. All these cubes intersect nicely to form a cube complex $C$ dual to the polyhedral decomposition.

If $P$ has some ideal vertices, the complement $M\setminus C$ consists of the open stars of the ideal vertices in the simplicial subdivision, without the ideal vertices: these are the cusps of $M$, and $M$ deformation retracts onto $C$.
\end{proof}

If $P \subset \matH^n$ has some ideal vertices, some more work is needed to enlarge the cube complex $C$ to a bigger cube complex $C^*$ whose interior is homeomorphic to $M$. This technical part is explained below and is necessary only in the cusped case, so it may be skipped at first reading.

\subsubsection{The cusped case} 
If $P\subset \matH^n$ has some ideal vertices, we consider it as a Euclidean polytope as above, and truncate all the ideal vertices to get a new Euclidean polytope $P^*$. The counterimage of $P^*$ in $M$ is a compact PL submanifold $M^*$ decomposed into copies of $P^*$, whose interior is homeomorphic to $M$. We get a polyhedral decomposition of $M^*$ into copies of $P^*$.

The truncation of every ideal vertex of $P$ produces a small $(n-1)$-cubic facet of $P^*$. In particular the decomposition of $M^*$ into copies of $P^*$ induces a decomposition of $\partial M^*$ into $(n-1)$-cubes.

As above, a \emph{$k$-face} of the decomposition of $M^*$ is a $k$-face of some copy of $P^*$. There are two types of $k$-faces: the \emph{boundary} ones, that are contained in $\partial M^*$, and the \emph{interior} ones, that are not. %So we have  boundary and interior vertices.

We fix a barycentric subdivision of $P^*$, which lifts to a barycentric subdivision of the polyhedral decomposition of $M^*$. We define $C^*$ to be the dual decomposition (defined as above) of the polyhedral decomposition of $M^*$ in this barycentric subdivision.

\begin{prop}
The construction produces a cube complex $C^*$ with $C^*=M^*$. 
There is a natural 1-1 correspondence between the $k$-faces of the polyhedral decomposition of $M^*$ and the $(n-k)$-cubes of $C^*$ that are not contained in $\partial M^*$, which reverses all containments.
The original $C$ is naturally a subcomplex of $C^*$.
\end{prop}
\begin{proof}
Let $F$ be a $k$-face of $M^*$. If $F$ is an interior face, its barycenter lies in the interior of $M^*$, and the dual face is a subdivided $(n-k)$-cube as above. If $F$ is a boundary face, its barycenter lies in $\partial M^*$. Similarly as above, we see that the dual cell is still naturally a $(n-k)$-cube, since it is a barycentrically subdivided $(n-k)$-cube cut in half.

The decomposition $C^*$ is naturally a cube complex, dual to the polyhedral decomposition of $M^*$ as stated. %The $(n-k)$-cubes of $C^*$ that are not contained in the boundary are dual to the initial polyhedral decomposition as stated. (The $(n-k)$-cubes that are contained in $\partial M$ are dual to the induced decomposition of $\partial M$.) 
By considering the faces that do not intersect the boundary we get the original subcomplex $C \subset C^*$.
\end{proof}

We call $C$ and $C^*$ the \emph{dual} and the \emph{extended dual} cube complexes.

\begin{example} \label{triangle:example}
Let $P\subset \matH^2$ be an ideal triangle. Its three edges are non-adjacent, so we can colour them with a single colour. The resulting manifold $M$ is obtained by mirroring $P$ along its whole boundary and is hence a thrice-punctured sphere. Here $C$ is a $\theta$-shaped graph, a spine of $M$. The compact polygon $P^*$ is a hexagon, $M^*$ is a compact pair-of-pants, and $C^*$ is a cube complex with 6 squares that contains $C$. Two squares are attached to each edge of $C$.
\end{example}

\subsection{The forgetful map}
We introduce a forgetful map from the cube complex $C$ to the standard $c$-cube that is quite simple and useful. Here $c$ is as usual the number of colours.

The $2^c$ vertices of $C$ are dual to the polytopes $P^v$ of the polyhedral decomposition of $M$, with $v \in \matZ_2^c$. We indicate the vertex of $C$ dual to $P^v$ with the same element $v \in \{0,1\}^c$. The forgetful map will send $v$ to $v$, considered as a vertex of $[0,1]^c$.

By construction a $k$-cube in $C$ is dual to a $(n-k)$-face $F$ of the polyhedral decomposition, which is contained in $2^k$ polytopes $P^v$ whose vectors $v$ differ only in some components $v_{i_1},\ldots, v_{i_k}$. The numbers $i_1,\ldots, i_k$ are precisely the colours of the facets in $P$ that contain the image of $F$.
This is a crucial fact that leads us to prove the following.

\begin{prop}
There is a canonical cube complex map from $C$ to the $c$-cube $[0,1]^c$ that sends $v$ to $v$. It sends $k$-cubes to $k$-cubes.
\end{prop}
\begin{proof}
Every vertex $w$ of the barycentric subdivision of $M$ is the barycenter of some face $F$. As we just said, the vectors $v$ of the $P^v$ adjacent to $F$ differ only at some coordinates $i_1,\ldots, i_k$. These vectors span a $k$-face of $[0,1]^c$. We send $w$ to the barycenter of this face. We get a simplicial map from $C$ to the barycentric subdivision of $[0,1]^c$, which is in fact a cube complex map (it sends $k$-cubes to $k$-cubes).
\end{proof}

We call it the \emph{forgetful map} because all the $k$-cubes in $C$ sharing the same vertices are sent to the same $k$-cube in $[0,1]^c$. It is neither injective nor surjective in general. 

\begin{example}
If $P$ is a 1-coloured ideal triangle as in Example \ref{triangle:example} the forgeftul map sends the $\theta$-shaped graph $C$ onto the segment $[0,1]$. 

If $P$ is the 2-coloured ideal octahedron $P$ the cube complex $C$ has four vertices indicated as $(0,0), (0,1), (1,0), (1,1)$, every vertex $(x,y)$ is connected to $(x,1-y)$ and $(1-x,y)$ with four edges each, and there are 12 squares in $C$, all with the same 4 vertices (but attached to different four-uples of edges). The forgetful map maps everything to a single square.
\end{example}

%Every edge of $C$ is dual to a $(n-1)$-face of the decomposition, which projects to some facet $F$ of $P$ coloured with some $i$. By construction the edge connects $v$ to $v'$, where the two vectors $v,v'$ differ only in their $i$-th coordinate. More generally, every $k$-cube of $C$ is dual to a $(n-k)$-face that projects to some $(n-k)$-face $F$ of $P$ adjacent to facets coloured with $i_0,\ldots, i_k$, and the vertices of the cube differ only in the coordinates $i_0,\ldots,i_k$. In fact the following holds.

\subsection{Real states and cocycles} \label{states:subsection}
Let again $P\subset \matX^n$ be a coloured right-angled polytope, producing a manifold $M$. Let $C$ be the dual cube complex. We extend the notion of state given in \cite{JNW} by defining a \emph{real state} $s$ to be the assignment of a real number $s(F)$ to each facet $F$ of $P$. The states considered in \cite{JNW} correspond to the case where one uses only the numbers $1$ and $-1$.

In the next lines we show that a real state $s$ on $P$ produces a cohomology class $[s] \in H^1(M, \matR)$ and a \emph{diagonal map} $f_s\colon \widetilde M \to \matR$ on the universal cover $\widetilde M$ of $M$. The whole construction is similar to \cite{JNW}.

Consider the forgetful map $C \to [0,1]^c$. Every edge of $[0,1]^c$ is oriented canonically like the coordinate axis parallel to it. The pull-back of these orientations give a \emph{canonical orientation} on the edges of $C$. On every $k$-cube of $C$ parallel edges are oriented coherently as in Figure \ref{arrows:fig}.

\begin{figure}
 \begin{center}
  \includegraphics[width = 7 cm]{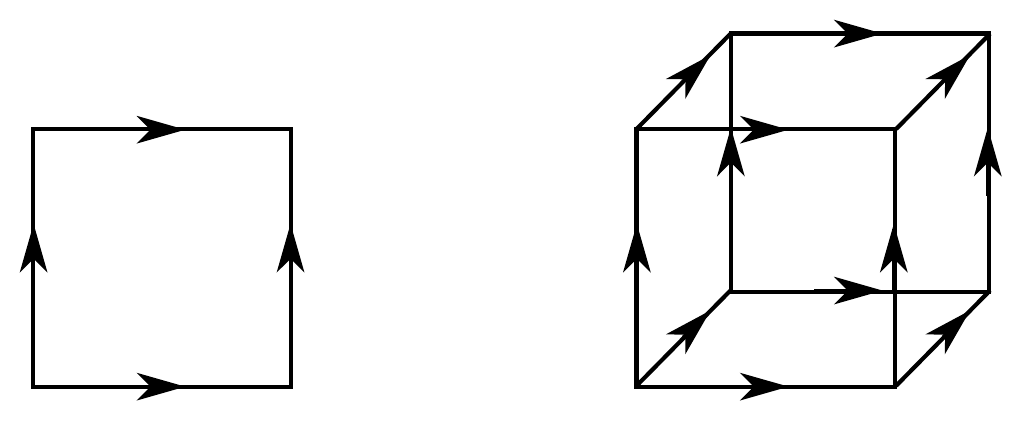}
 \end{center}
 \nota{Parallel edges of every $k$-cube are cooriented.}
 \label{arrows:fig}
\end{figure}

Let $s$ be a real state for $P$. We assign to every oriented edge $e$ of $C$ the real number $s(e) = s(F)$ where $F$ is the facet of $P$ that is the image of the $(n-1)$-face of the polyhedral decomposition dual to $e$.
By assigning a real number to every oriented edge $e$ of $C$ we have just defined a cellular 1-cochain, that we still denote by $s$. 

\begin{prop} \label{cocycle:prop}
The cellular cochain $s$ is a cocycle. 
\end{prop}
\begin{proof}
We need to prove that the sum of the contributions on every square $Q$ of $C$ is zero. The edges of $Q$ are oriented as in Figure \ref{arrows:fig}. The square $Q$ is dual to a $(n-2)$-face $F$ of the polyhedral decomposition, which projects to a $(n-2)$-face of $P$, that is in turn the intersection of two facets $F_1$ and $F_2$ of $P$. The four edges of $Q$ are dual to the four $(n-1)$-faces of the decomposition containing $F$. By construction, these four faces project alternatively to $F_1$ and $F_2$. Two opposite edges $e,e'$ of $Q$ are dual to faces that project to the same facet $F_i$ and hence $s(e)=s(e')$. Their contributions to the sum is zero since they are oriented as in Figure \ref{arrows:fig} (with opposite directions with respect to the canonical orientation of $\partial Q$). So the overall contribution on $Q$ is zero and $s$ is a cocycle.
\end{proof}

Every real state $s$ determines a cohomology class $[s] \in H^1(C,\matR) = H^1(M,\matR)$. (Remember that $M$ deformation retracts onto $C$.) We can characterise easily which states yield the same classes:

\begin{prop}
Let $s,s'$ be two real states. We have $[s] = [s']$ if and only if there is a $\lambda = (\lambda_1,\ldots, \lambda_c)\in \matR^c$ such that $s'(F) = s(F) + \lambda_i$ where $i$ is the colour of $F$, for every facet $F$ of $P$.
\end{prop}
\begin{proof}
We suppose that here is a $\lambda\in\matR^c$ that fulfills the requirement, and we show that $[s] = [s']$. It suffices to consider the case where $\lambda_j = 0$ for all $j$ distinct from some fixed $i$, the general case will follow by iterating.  Consider the 0-cochain $\alpha$ that assigns $\lambda_{i}$ to all the vertices $v$ with $v_{i}=0$ and 0 to the others. One can verify that $s' = s + d\alpha$ and hence $[s] = [s']$. 

Conversely, we suppose that $[s] = [s']$ and prove that there is a $\lambda$ that fulfills the requirement. We identify $P^0$ with $P$. Two facets $F_1, F_2$ of $P^0=P$ sharing the same colour $i$ are dual to two edges of $C$ connecting $P^0$ to $P^{e_i}$. These two edges form a 1-cycle $c$ (after reversing the orientation of the second edge), and we have $s(c) = s(F_1)-s(F_2)$. Since $s(c)$ depends only on $[s]$, we deduce that $s(c) = s'(c)$. In particular we may define unambiguously $\lambda_i = s'(F) - s(F)$ by taking any facet $F$ with colour $i$, and the resulting $\lambda$ is as required.
\end{proof}

A real state $s$ is \emph{balanced} if the sum of all the $s(F)$ among the facets $F$ sharing the same colour is zero, for every colour. We can deduce easily the following.

\begin{cor}
For every state $s$ there is a unique balanced state $s'$ with $[s'] = [s]$.
\end{cor}

A balanced state is like a combinatorial harmonic representative.
The balanced states form naturally a subspace of $H^1(C,\matR) = H^1(M,\matR)$. This subspace actually coincides with $H^1(M,\matR)$ on the three manifolds $W, X, Z$ that we construct in this paper, but this does not hold in general: there may be cohomological elements that are not represented by states.

\begin{example}
Consider the $n$-cube with its $n$-colouring, producing the $n$-torus $M$. A balanced state assigns two opposite numbers to each opposite facets. The balanced states form a vector space of dimension $n$, naturally identified with $H^1(M, \matR)$. If we use more than $n$ colours, we still get a flat $n$-torus, but the balanced states form only a proper subspace of $H^1(M, \matR)$.

Consider the ideal regular hyperbolic octahedron with its 2-colouring, producing the complement $M$ of the chain link in Figure \ref{chainlink:fig}. A balanced state assigns two 4-uples of real numbers with sum zero. The balanced states form a space of dimension 6, naturally isomorphic to $H^1(M, \matR)$.
\end{example}

\subsubsection{The cusped case}
If $P\subset \matH^n$ is cusped, we may need to extend the cocycle $s$ to the extended cubulation $C^*$, so that both $s$ and its extension represent the same element in $H^1(C,\matR) = H^1(C^*,\matR) = H^1(M,\matR)$. There is a canonical way to do it that we now explain.

There are two types of vertices in $C^*$, the \emph{boundary vertices} that are contained in $\partial M^*$, and the \emph{interior vertices} that are not: the latter ones are precisely the vertices of $C$.
Correspondingly, there are two types of edges in $C^* \setminus C$, those that connect an interior to a boundary vertex and those that connect two boundary vertices. We orient every edge $e$ of the first type towards the boundary, and we set $s(e)=1$. An edge $e$ of the second type is parallel to some edge of $C$ and hence inherits from it both the canonical orientation and the real number $s(e)$. It is easy to verify that this extension is a cocycle in $C^*$ that represents the same cohomology class in $H^1(M,\matR)$.

\subsection{Diagonal maps} \label{diagonal:subsection}
Let $s$ be a real state on $P$. We now suppose that $s$ is nowhere vanishing, that is $s(F)\neq 0$ for all the facets $F$ of $P$. This allows us to modify the orientation of the edges of $C$ and $C^*$, by reversing the canonical orientation of $e$ if and only if $s(e)<0$. We call the resulting orientation the \emph{orientation induced by $s$}. Parallel edges in every $k$-cube are still cooriented, because they were all assigned the same real number (see the proof of Proposition \ref{cocycle:prop}). We fix henceforth this orientation, and modify accordingly the signs of all the labels $s(e)$ so that $s(e)>0$ on every edge $e$. Orientations and labels have changed but of course $s$ identifies the same cohomology class as before.

We identify inductively each $k$-cube of $C$ or $C^*$ with the standard Euclidean $k$-cube $[0,1]^k$ in an edge-orientation-preserving way: that is, the unique vertex from which all edges are departing is sent to $0$. The identification is meant to be fixed only up to permuting the $k$ coordinates. 

We lift the orientation of the edges, the cocycle $s$, and the identifications with the Euclidean cubes from the cube complex $C$ to its universal cover $\widetilde C$. The state $s$ induces a \emph{diagonal map} $f\colon \widetilde C \to \matR$ as follows. Choose a basepoint vertex $v_0$ in $\tilde C$ and set $f(v_0) = 0$. Extend this map to every $k$-cube via the diagonal map
$$f(x_1, \ldots, x_k) = f(0) + s(e_1)x_1 + \cdots + s(e_k)x_k$$ 
where $e_i$ is the edge contained in the $i$-th axis. The constant $f(0)$ is chosen inductively on each $k$-cube so that this map matches with the previously assigned cubes. There is no ambiguity since $\tilde C$ is simply connected.

The resulting $f\colon \widetilde C\to \matR$ is a Morse function, in the sense of \cite{BB}. Using the basepoint $v_0$ and the identification of $\pi_1(C)$ with the deck transformations group, the map $f$ induces a homomorphism $\pi_1(C) \to \matR$ and hence a class in $H^1(M; \matR)$ that is equal to $[s]$. If $[s] \in H^1(M;\matZ)$  the class has integral periods and therefore $f$ descends to a map $f\colon C \to \matR/_{\matZ} = S^1$.

In the cusped case we actually work with $C^*$ instead of $C$. In both cases we get a map $f\colon \widetilde M \to \matR$ on the universal cover $\widetilde M$, called a \emph{diagonal map}, which descends to a map $f\colon M \to \matR/_\matZ = S^1$ if the class is integral. We sometimes denote it as $f_s$ to stress the dependence on $s$.

\subsection{The ascending and descending links} \label{ad:subsection}
As we have seen, a nowhere vanishing real state $s$ on $P$ defines a diagonal map $f\colon \widetilde M \to \matR$. We can now apply Bestvina -- Brady theory to study this map topologically.

Let $v$ be a vertex of $C$ (or $C^*$ in the cusped case). Recall that the edges of $C$ (or $C^*$) are oriented by $s$. Let $\lk(v)$ be the link of $v$ in $C$ (or $C^*$). By construction $\lk(v)$ is an abstract simplicial complex homeomorphic to $S^{n-1}$ (or $D^{n-1}$ if $v$ is a boundary vertex of $C^*$). Every vertex of $\lk(v)$ indicates an oriented edge of $C$ (or $C^*$) incident to $v$, and we assign to it the \emph{status} I (In) or O (Out) according to whether the oriented edge points towards $v$ or away from $v$. 

Following \cite{BB}, we define the \emph{ascending link} $\lk_\uparrow(v)$ (respectively, \emph{descending link} $\lk_\downarrow(v)$) to be the subcomplex of $\lk(v)$ generated by all the vertices with status O (respectively, I).

In the following we suppose that $\dim P \leq 4$, as it will certainly be the case here. The condition stated below must hold at every $v$ in $C$ (or $C^*$ in the cusped case). We refer to \cite{RS} for the PL notions of collapse and of $k$-dimensional polyhedron.

\begin{teo} \label{ad:teo}
Suppose that at every $v$ both the ascending and descending links collapse to a connected polyhedron of dimension $\leq 1$, and to a point if $v$ is a boundary vertex. Then $f\colon \widetilde M \to \matR$ can be smoothened to a Morse function with only critical points of index 2.

If the ascending and descending links collapse to points at every $v$, then $f$ can be smoothened to a fibration.
\end{teo}
\begin{proof}
For a subspace $J\subset \matR$, we define $\widetilde M_J = f^{-1}(J)$. We write $\widetilde M_t = M_{\{t\}}$. If an interval $J \subset \matR$ contains no image of a vertex of $\widetilde C$ (or $\widetilde C^*$), then $\widetilde M_J$ is a union of prisms and hence PL homeomorphic to a product submanifold $M_t \times J$, for any $t \in J$. 

Suppose that $t\in \matR$ is the image of some vertex $v$. We suppose for the sake of simplicity that no other vertex is sent to $[t-\varepsilon, t+\varepsilon]$. The following argument is of local nature and easily extends to the general case.

As shown in \cite[Lemma 2.5]{BB}, the submanifold $\widetilde M_{(-\infty,t+\varepsilon]}$ collapses onto the union $\widetilde M_{(-\infty,t-\varepsilon]} \cup C_v$ where $C_v$ is the cone over the simplicial complex $\lk_\downarrow(v)$, considered inside $\widetilde M_{t-\varepsilon}$. 

By hypothesis $\lk_\downarrow(v)$ collapses either to a point or to a 1-dimensional polyhedron in $\widetilde M_{t-\varepsilon}$. In the first case
the manifold $\widetilde M_{[a,t+\varepsilon]}$ collapses to $\widetilde M_{[a,t-\varepsilon]}$, and hence $\widetilde M_{[t-\varepsilon, t+ \varepsilon]}$ is PL homeomorphic to a product $M_t \times [-\varepsilon, + \varepsilon]$. In the second case $\widetilde M_{[a,t+\varepsilon]}$ collapses to $\widetilde M_{[a,t-\varepsilon]}$ with a cone over a 1-polyhedron in $\widetilde M_{t-\varepsilon}$. Therefore $\widetilde M_{[a,t+\varepsilon]}$ is PL homeomorphic to $\widetilde M_{[a,t-\varepsilon]}$ with some $2$-handles attached. Recall that this is allowed only when $v$ is not a boundary vertex.

By what just proved, the function $f$ determines a PL handle decomposition for $\widetilde M$, where we start with any regular fiber and proceed both in the positive and in the negative direction by attaching 2-handles whenever we cross a vertex $v$ whose ascending or descending link collapses onto a non-simply connected 1-polyhedron. Since we are in low dimension $n\leq 4$, this handle decomposition can be smoothened \cite{HM, M}. The smooth handle decomposition can finally be transformed into a Morse function with only index 2 critical points.
\end{proof}

\begin{rem}[Integral class] 
If $[s]$ is integral, the smoothenings can be done equivariantly so that $f$ descends to a cicle-valued Morse function $f\colon M \to \matR/_\matZ = S^1$ with only critical points of index 2.
\end{rem}

\begin{rem} [The algorithm] \label{algorithm:rem}
We describe a simple algorithm to compute the ascending and descending link at every vertex $v$ of $C$ when $P$ is compact. The link of $v$ in $C$ is a simplicial complex dual to $P^v$. Every vertex $w$ of the link corresponds dually to a facet $F$ of $P^v$. 

Recall that the vertices of $C$ are identified with $\matZ_2^c$.
At $v=0$, the status of a vertex $w$ of the link is O and I depending on whether $s(F)$ is positive or negative. Whenever we pass from $v$ to $v+e_i$, the stati of the vertices dual to the facets $F$ coloured  with $i$ are switched, while all the others stay the same. Via these simple moves we can determine the stati and hence the ascending and descending links of all the vertices $v\in \matZ_2^c$. This is the basis for the combinatorial game of \cite{JNW}.
\end{rem}

In what follows, a \emph{hyperoctahedron} is the $(n-1)$-dimensional simplicial complex dual to the boundary of a $n$-cube. It is of course homeomorphic to $S^{n-1}$.

\begin{example}[$n$-cubes] \label{cubes:example}
Let $P$ be a $n$-cube with some colouring. This produces a flat $n$-torus $M$. Let $s$ be a nowhere vanishing state for $P$. If there are two opposite facets $F_1, F_2$ of $P$ with the same colour and $s(F_1)$, $s(F_2)$ have opposite signs, then all the ascending and descending links collapse to points and $f_s$ is a fibration. 

To prove this, note that the link of a vertex $v$ of $C$ is a hyperoctahedron. Using the previous remark we deduce that at every $v$ we have two opposite vertices of the hyperoctahedron with opposite status I and O. This implies easily that both the ascending and the descending links collapse to these opposite vertices, that is to points.
\end{example}

\subsubsection{The cusped case} 
Let $P\subset \matH^n$ have some ideal vertices, equipped with a colouring and a nowhere vanishing real state $s$. We now expose a simple criterion that, when verified, allows us to forget about the boundary vertices, and to use $C$ instead of $C^*$ for the interior vertices.

\begin{prop} \label{ideal:prop}
If every ideal vertex of $P$ is adjacent to two facets $F_1, F_2$ with the same colour and with $s(F_1), s(F_2)$ having opposite signs, the hypothesis of Theorem \ref{ad:teo} is satisfied at every boundary vertex of $C^*$. Moreover, to check the hypothesis on each interior vertex $v$, it suffices to consider the link of $v$ in $C$ instead of $C^*$.
\end{prop}
\begin{proof}
The link in $C^*$ of a boundary vertex is a cone over a hyperoctahedron. As in Example \ref{cubes:example}, the hypothesis implies that there are two opposite vertices of the hyperoctahedron with opposite status, and from this we can easily deduce that the ascending and descending links both collapse to points.

The $v$ be an interior vertex of $C^*$. We now prove that it suffices to consider its link in $C$ instead of $C^*$. The link of $v$ in $C$ is dual to $P$, with a hole corresponding to every ideal vertex of $P$. Indeed the link is homeomorphic to $S^{n-1}$ minus some open balls (the holes). The boundary of each hole is a hyperoctahedron. The link of $v$ in $C^*$ is obtained by filling all these holes, thus getting a simplicial complex homeomorphic to $S^{n-1}$. Each hole is filled by adding a cone over the corresponding hyperoctahedron, with a new central vertex $w$ that indicates an edge connecting $v$ to a boundary vertex of $C^*$. The vertex $w$ has status O because we have oriented all these edges towards the boundary.

The descending links of $v$ in $C$ and $C^*$ are the same. The ascending link in $C^*$ is obtained from that in $C$ by adding the new central vertices $w$ at each hyperoctahedron. By hypothesis there are two opposite vertices in each hyperoctahedron with opposite status. As already noticed, this implies that the link of $w$ in the ascending link collapses to a point. Therefore the ascending link in $C^*$ collapse to the ascending link in $C$. So it suffices to prove Theorem \ref{ad:teo} using the link of $v$ in $C$ instead of $C^*$.
\end{proof}

One may in fact prove that this condition is also necessary, but we will not need that. Summing up: if the stated condition near the ideal vertices of $P$ is verified (facets with the same colours have opposite signs), we can forget about $C^*$ and we need only to determine the ascending and descending links of all the vertices of $C$ inside $C$. This can then be done using Remark \ref{algorithm:rem}, as in the compact case.

\section{The manifolds} \label{manifolds:section} 
We define here the manifolds $W, X, Z$, the orbifold $Y$, and some perfect circle-valued Morse functions on them. 

Many of the calculations exposed here have been carried out with a program written in Sage that is publicly available from \cite{codeM}. The program takes as an input the adjacency matrix of a right-angled polytope $P$, a colouring of its facets, and a state that assigns a value $\pm 1$ at each facet. It determines the Betti numbers and the cusps of the manifold $M$ obtained from $P$ and its colours, and verifies whether the state fulfills the hypothesis of Theorem \ref{ad:teo}. It also constructs a triangulation of the (singular) fibers, in a format that can be read by SnapPy \cite{Sna} and Regina \cite{BBP}.

\subsection{The manifold $W$} \label{W:subsection}
We now define $W$ and prove Theorem \ref{W:teo}. As already mentioned, there is a sequence of remarkable right-angled polytopes $P_3,\ldots, P_8$ already considered by various authors \cite{ALR, PV, ERT}. 

We are interested here in $P_4$. The polytope is clearly presented in \cite{RT4}. It has 10 facets, 5 real vertices, and 5 ideal vertices. Every facet is isometric to the polyhedron $P_3$, a right-angled bipyramid with two real vertices and three ideal ones. Its orbifold Euler characteristic is $\chi(P_4) = \frac 1{16}$.

\begin{figure}
 \begin{center}
  \includegraphics[width = 10 cm]{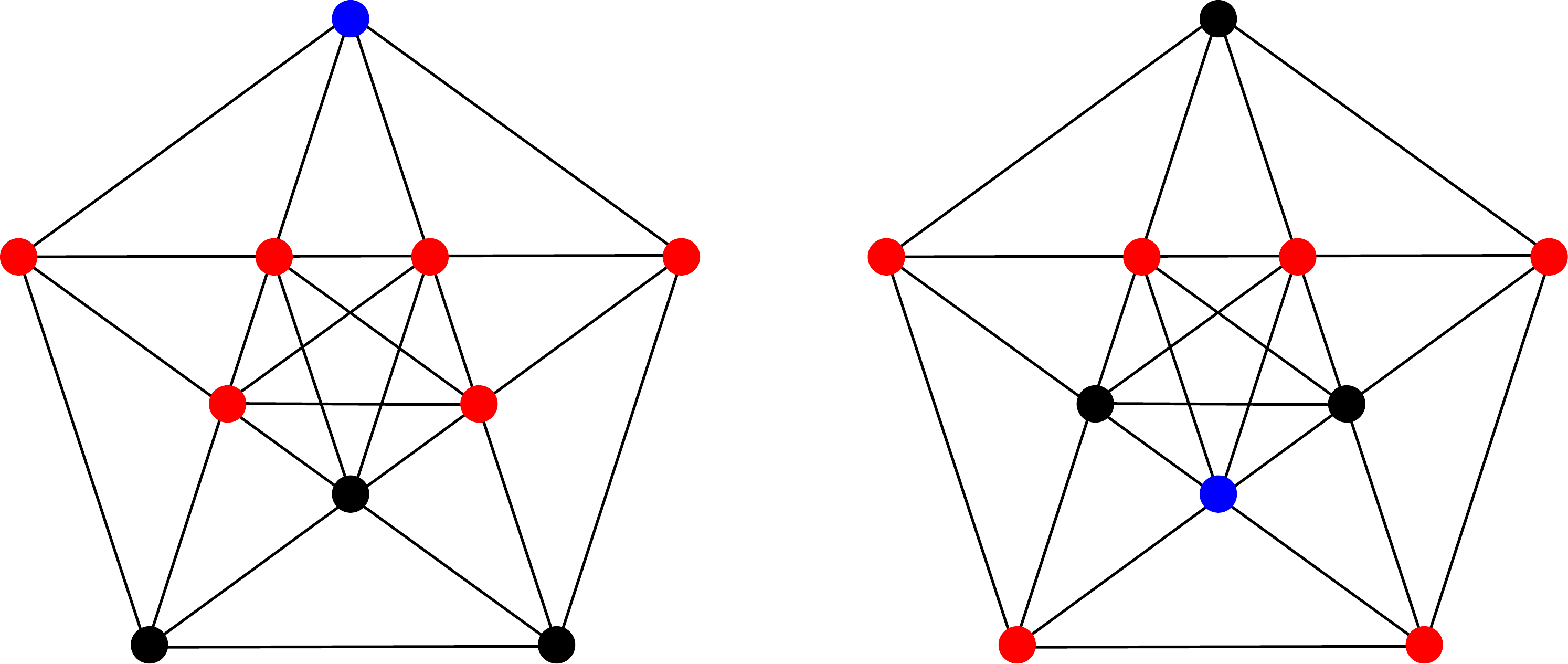}
 \end{center}
 \nota{The adjacency graph of the 10 facets of $P_4$. Some edges are superposed for simplicity, so some caution is needed here: to clarify this ambiguity, we have chosen a blue vertex and painted in red the 6 vertices adjacent to it and in black those that are not adjacent, in two cases (all the other cases are obtained by rotation).}
 \label{P4:fig}
\end{figure}

\begin{figure}
 \begin{center}
  \includegraphics[width = 5 cm]{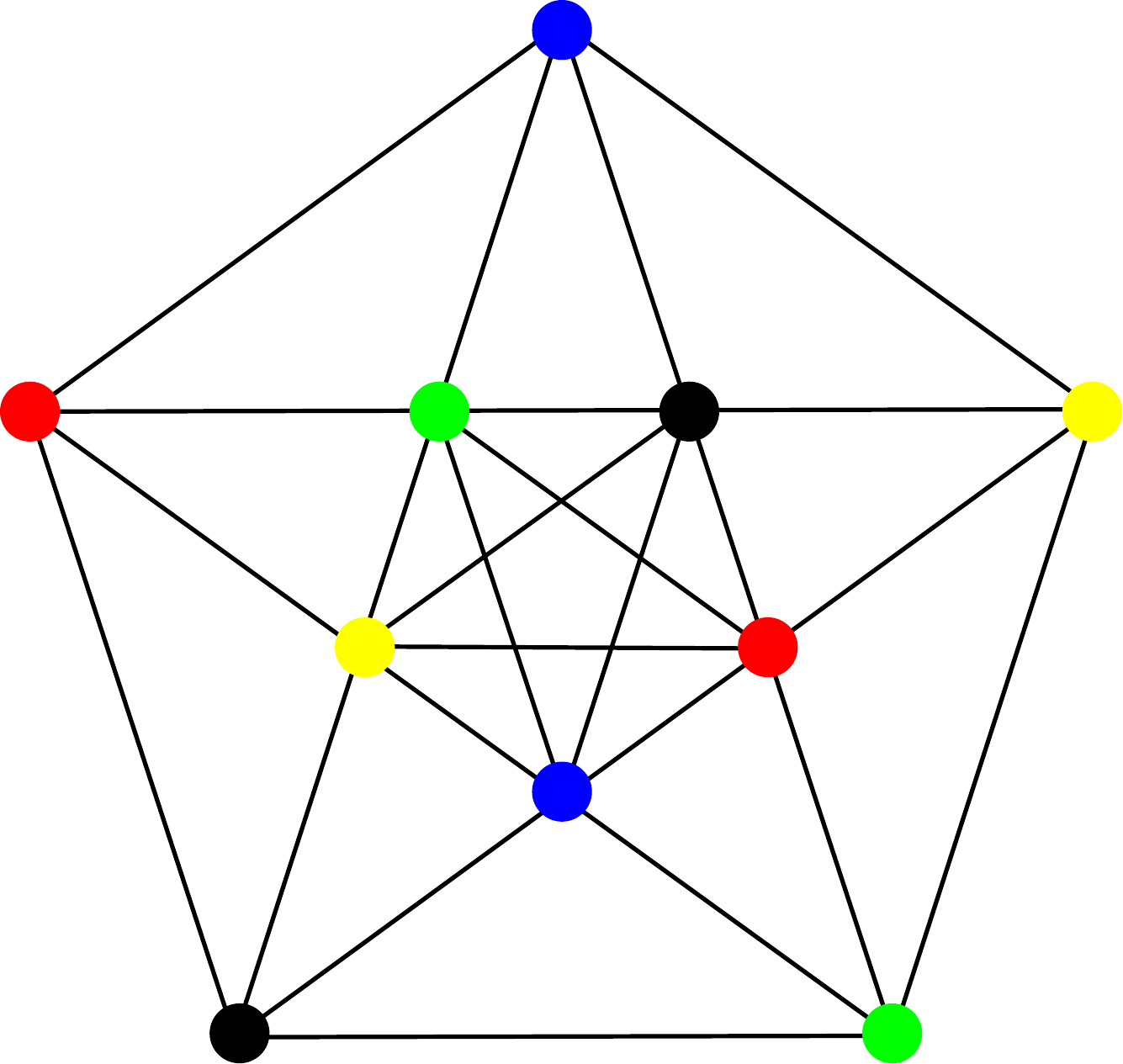}
 \end{center}
 \nota{A 5-colouring for $P_4$.}
 \label{P4_coloured:fig}
\end{figure}

The adjacency graph of its facets is shown in Figure \ref{P4:fig}. We assign to $P_4$ the 5-colouring shown in Figure \ref{P4_coloured:fig}. The colouring construction described in 
Section \ref{colours:subsection} produces a hyperbolic cusped orientable 4-manifold $W$, tessellated into $2^5$ copies of $P_4$ and hence with $\chi(W) = 2$. 

At every ideal vertex $v$ of $P_4$ one can verify that the Euclidean cube cross-section inherits from $P_4$ a colouring with all the 5 colours involved, one of which appears twice in opposite faces of the cube. Such a colouring yields a flat 3-torus that decomposes into $2^5$ cubes. The hyperbolic manifold $W$ has 5 cusps, one for each ideal vertex of $P_4$. Each cusp has a 3-torus section.

Using Sage (or by hand) we find that $H^1(W; \matR) = \matR^5$, and the cohomology class $x\in \matR^5$ is described by the balanced real state $s$ that assigns $x_i/2$ and $-x_i/2$ to the two facets that share the $i$-th colour, as explained in Section \ref{states:subsection}. The factor $1/2$ is there only to ensure that the lattice $H^1(W; \matZ)$ corresponds to $\matZ^5$. This state $s$ defines a diagonal map $f_s$ on the universal cover of $W$, as explained in Section \ref{diagonal:subsection}.

\begin{figure}
 \begin{center}
  \includegraphics[width = 12.5 cm]{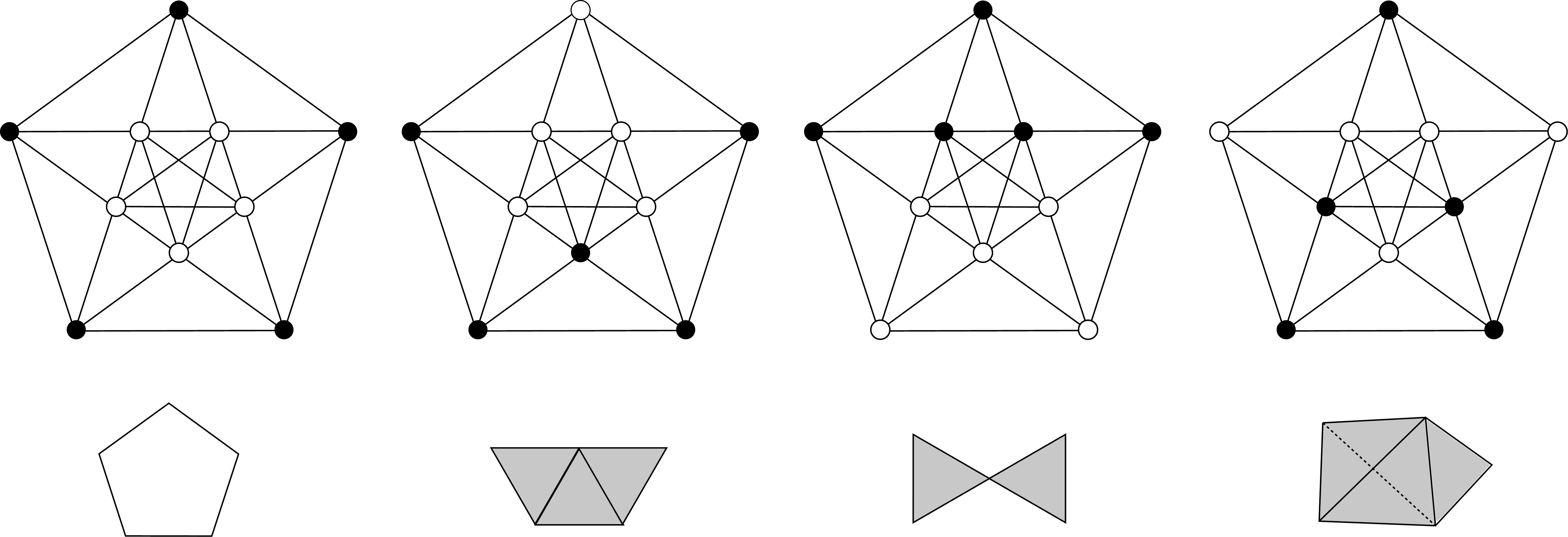}
 \end{center}
 \nota{The ascending and descending links for $P_4$, up to isomorphisms. A black (white) vertex has status I (O). The descending link, generated by the black vertices, is shown below. In the first case we get a circle, while in the other cases we always get a contractible complex made of 3 triangles, 2 triangles, and 1 tetrahedron and 1 triangle respectively. The ascending links are the same (ordered differently).}
 \label{P4_states:fig}
\end{figure}

\begin{prop}
If $x_i \neq 0$ for all $i$, then $f_s$ can be smoothened to a perfect circle-valued Morse function with two index-2 singular points.
\end{prop}
\begin{proof}
We show that the hypothesis of Theorem \ref{ad:teo} are fulfilled. The criterion of Proposition \ref{ideal:prop} is verified: every ideal vertex is adjacent to two facets with the same colour, and $s$ has opposite values on these by assumption. Hence we can use $C$ instead of $C^*$.

To identify the ascending and descending links at the vertices $v\in \matZ_2^5$ of $C$ we follow Remark \ref{algorithm:rem}. We get the simplicial complexes shown in Figure \ref{P4_states:fig}. They all collapse either to a point or to a circle.

There are two index-2 singular points: we find this either by inspection, or by recalling that $\chi(W)=2$.
\end{proof}

This completes the proof of the constructive part of Theorem \ref{W:teo}. To conclude we need to show that if $x_i=0$ then there exists no perfect circle-valued Morse function. The reason is quite simple: if $f\colon W \to S^1$ represents a cohomology class with $x_i=0$, we can prove that its restriction to the $i$-th cusp of $W$ is homotopically trivial. Therefore $f$ cannot be a fibration there, and it is not homotopic to any circle-valued Morse function of any kind.

\subsubsection{The singular and regular fibers} \label{sr:subsubsection}
We have proved the remarkable fact that each nowhere-vanishing $x\in H^1(W,\matZ) = \matZ^5$ is represented by a circle-valued Morse function $f\colon W \to S^1$ with two critical points of index 2. We now would like to understand these functions $f$ topologically. 
The isometries of $W$ act by switching the signs of each coordinate $x_i$ arbitrarily, so we can always suppose that $x$ lies in the first orthant $x_i>0$. Using Sage we have studied the case $x=(1,1,1,1,1)$. We now expose the topological information that we have found in this case.

The function $f\colon W \to \matR/\matZ = S^1$ has two critical points, with values $0$ and $1/2$. We have two singular fibers $f^{-1}(0)$ and $f^{-1}(1/2)$ and two regular fibers $f^{-1}(1/4)$ and $f^{-1}(3/4)$. Each singular fiber is a 3-manifold with 6 boundary tori $W^{\rm sing}$, with one boundary torus coned to a point. The two nearby regular fibers $W^{\rm reg}$ have 5 boundary tori (one inside each cusp of $W$) and are obtained by substituting this cone point with a circle in two different ways: the two regular fibers are obtained one from the other via integral Dehn surgery along a knot.

We have used Sage to identify all these fibers. Both singular fibers are the same hyperbolic manifold $W^{\rm sing}$ with 6 cusps (one of which is coned) and volume $\sim$ 65.318656269. It has a geometric triangulation with 72 tetrahedra. Both the regular fibers are the same hyperbolic manifold $W^{\rm reg}$ with 5 cusps and volume $\sim 54.991958042$. It has a geometric triangulation with 70 tetrahedra. The fact that we get twice the same manifolds is probably due to the symmetries of $W$.

Using SnapPy \cite{Sna} we can see that indeed $W^{\rm reg}$ can be obtained from $W^{\rm sing}$ by filling one cusp along two different slopes at distance 1 -- this phenomenon is called \emph{cosmetic surgery} \cite{BHW}. There is an isometry $\psi$ of $W^{\rm sing}$ that sends the first slope to the second, and this explains why we get the same filled manifold. 

SnapPy also tells us a fact that seems relevant to us: the filled $W^{\rm reg}$ has an additional isometry $\varphi$ that is \emph{not} induced from $W^{\rm sing}$, and that acts homologically trivially on the cusps.

\begin{rem}[Ping pong] \label{ping:pong:rem}
Although we did not prove it (the combinatorics involved is non-trivial), we believe that the map $f$ should have the following dynamical behaviour, that may be described as a \emph{ping-pong} between the isometries $\psi$ and $\varphi$. We start with the regular fiber $W^{\rm reg}$ at some time $t$. We increase $t$ and when we cross a critical value the fiber is surgered along a knot, whose complement is $W^{\rm sing}$, to get a new manifold that is in fact diffeomorphic to the original $W^{\rm reg}$ via $\psi$. Now we act on $W^{\rm reg}$ via $\varphi$ and repeat the process from the beginning.

The role of $\varphi$ should be fundamental to ``mix everything up'' like in the familiar pseudo-Anosov monodromies: since $\varphi$ is an additional isometry, not induced from one on $W^{\rm sing}$, we are not just surgerying all the time along the same knot -- that picture would be too simple and could not hold because it would produce some non-peripheral $\matZ \times \matZ$ inside $\pi_1(W)$, like with the reducible mapping classes in dimension $2+1=3$. In some vague sense, the two isometries $\psi$ and $\varphi$ look like the matrices $\matr 1101$ and $\matr 1011$, that when multiplied give rise to $\matr 2111$, the Anosov monodromy of the famous fibration of the figure-8 knot complement, fully guaranteed not to fix any non-peripheral closed curve on the punctured torus.
\end{rem}

\subsubsection{The abelian cover: strong numerical evidence for its rigidity}\label{abeliancover:subsection}
We keep studying the case $x=(1,1,1,1,1)$.
Consider the abelian cover $\widetilde W$ induced by $\ker f$. It is a geometrically infinite manifold with limit set $S^3$, obtained topologically from $W^{\rm reg} \times [0,1]$ by attaching infinitely many 2-handles on both sides. As noted in the introduction, the fundamental group of $\widetilde W$ is finitely generated but not finitely presented, since $b_2(\widetilde W) = \infty$.

It is natural to ask whether the hyperbolic structure on the geometrically infinite $\widetilde W$ is rigid. One usually start by investigating whether the holonomy representation is infinitesimally rigid. Although this is a linear problem, it is not straightforward because $\dim \Iso(\matH^4)=10$ and $\pi_1(\widetilde W)$ has many generators and infinitely many relators. We have elaborated a general strategy that applies to abelian covers constructed from right-angled polytopes, colours, and states, and written a separate computer program. We found very strong numerical evidence for the infinitesimal rigidity of the holonomy of $\widetilde W$ and of any other similar abelian cover investigated in this paper. This will be explained in a forthcoming paper \cite{Bat}. We proved the infinitesimal rigidity rigorously in one case in Section \ref{very:symmetric:subsubsection} below.

The strong numerical evidence for the infinitesimal rigidity of $\widetilde W$ and other abelian covers is in contrast to dimension three, where as a consequence of the (now proved \cite{BrBr, B, NS}) density conjecture we know that every geometrically infinite 3-manifold can be deformed into a geometrically finite one. This contrast is however not surprising: it is natural to experience more rigidity in dimension $n\geq 4$, especially for a manifold that has infinitely many 2-handles.  Nevertheless, these seem to be the first examples of rigid geometrically infinite hyperbolic manifolds in any dimension. See \cite{KS} for a related rigidity result (that however does not apply in this case).

%The fundamental group of $\widetilde W$ has finitely many generators and infinitely many relators. To achieve the infinitesimal rigidity we considered only a reasonable amount of these relators, and we briefly explain how. We consider the lifted cubulation $\widetilde C \subset \widetilde W$. The diagonal map $f$ sends every vertex in $\widetilde C$ to some half-integer $m \in \matZ/2$, every edge to an interval $[m,m+1/2]$, and every square to an interval $[m,m+1]$. For every pair $m < n$ of half-integers we define $\widetilde C_{[m,n]}$ as the union of all the (finitely many) squares whose image lies in $[m,n]$. If $n-m \geq 1$, this is a non-empty subset that contains the 1-skeleton of a singular fiber, which generates $\pi_1(\widetilde C)$. So in particular $\pi_1(\widetilde C_{[m,n]})$ is a finitely presented group that naturally surjects onto $\pi_1(\widetilde C)$.

%In general, by considering a wider interval $[m,n]$ we include more relators and hence potentially get more rigidity for the holonomy of $\widetilde C_{[m,n]}$, that is the dimension of the space of all infinitesimal deformations either stays the same or strictly decreases when we enlarge $[m,n]$. For the manifold $\widetilde W$ that we are considering here, we discover (via computer) that the holonomy of $\pi_1(\widetilde C_{[0,1]})$ is already infinitesimally rigid, and hence $\widetilde W$ also is. The code is publicly available in \cite{code}. 
%We will find below some examples where the infinitesimal rigidity is not achieved at the first step, see Section \ref{very:symmetric:subsubsection}.

\subsection{The manifold $X$} \label{X:subsection}
We now construct one more cusped example, via another right-angled polytope, the ideal 24-cell $\calC \subset \matH^4$. This regular polytope has 24 facets, each being a right-angled ideal regular octahedron. It has a unique 3-colouring, shown in Figure \ref{colori_numeri:fig}. This produces a very symmetric cusped hyperbolic 4-manifold $X$ that was already considered in \cite{KM, MR}. We have $\chi(\calC) = 1$ and $\chi(X) = 8$.

\begin{figure}
 \begin{center}
  \includegraphics[width = 11 cm]{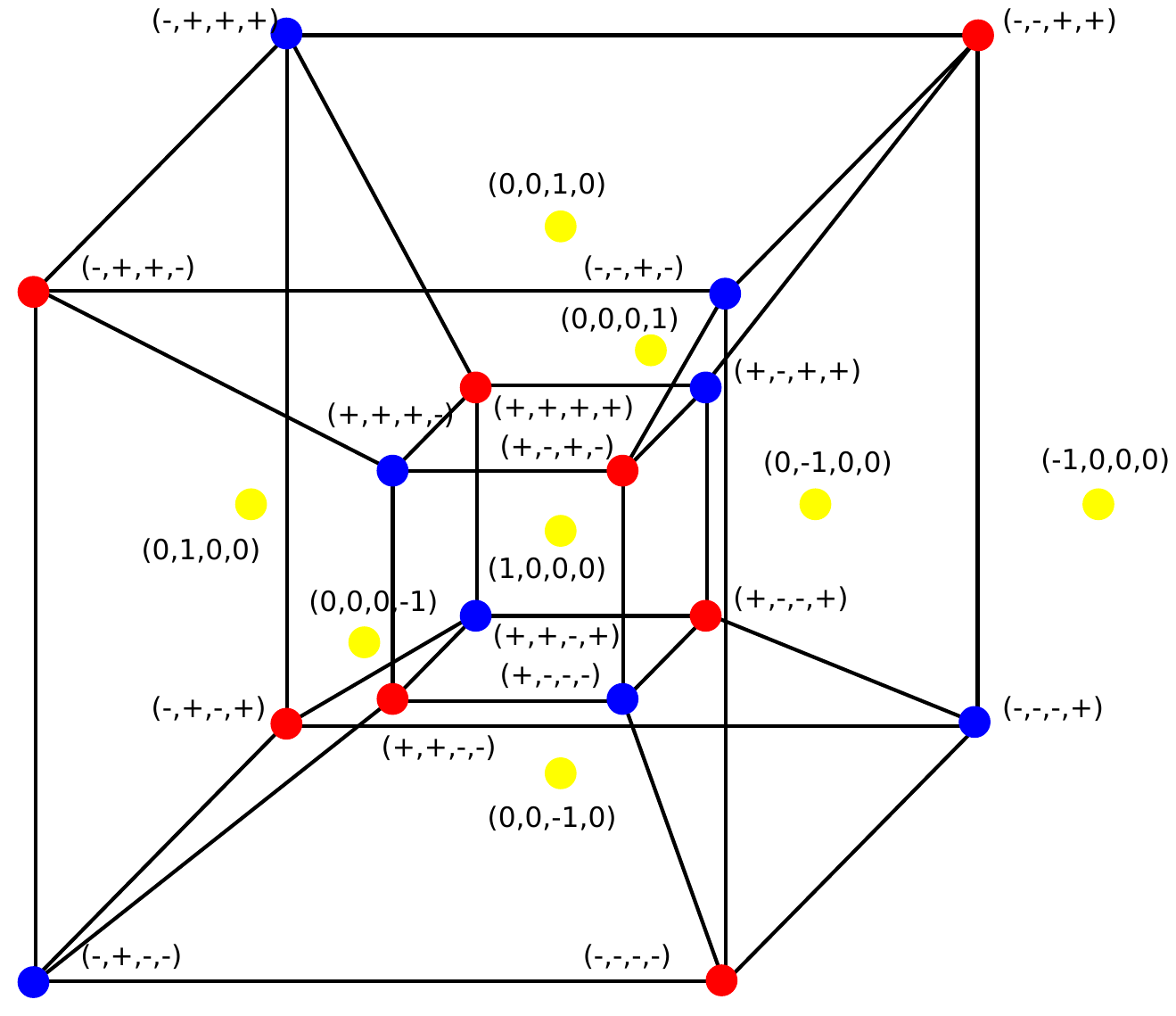}
 \end{center}
 \nota{The picture shows the dual $\calC^*$ of the 24-cell $\calC$, which is again a 24-cell. Not all the edges are drawn, for simplicity: every yellow vertex is the center of a cube and we should add 8 more edges connecting it with the vertices of the cube.
 The vertices of $\calC^*$ are 3-coloured. The picture is taken from similar figures in \cite{JNW}.}
 \label{colori_numeri:fig}
\end{figure}

At every ideal vertex of $\calC$ we have a Euclidean cube coloured with three colours, that gives rise to a 3-torus cusp section. The manifold $X$ has 24 such cusps, one for each ideal vertex of $\calC$.

\begin{rem}[Quaternions]
The colouring of $\calC$ can be described in the following way. Consider $\calC$ inside the quaternions space $\matH$. The dual of $\calC$ is another 24-cell $\calC^*\subset \matH$, whose vertices form the binary tetrahedral group $T_{24}^*$. Now $Q_8= \{\pm 1, \pm i, \pm j, \pm k\}$ is a normal subgroup $Q_8 \triangleleft T_{24}^*$ of index 3. The three lateral classes of $Q_8$ subdivide the 24 vertices of $\calC^*$ into three octets, and we assign a colour to each octet. 
\end{rem}

The existence of various states for $\calC$ that give rise to connected ascending and descending links was first proved in \cite{JNW}. The original goal of our work was to better understand the topology of the resulting algebraically fibering maps. As in \cite{JNW}, we let a \emph{state} be a real state $s$ that assigns the number $\pm 1$ to each facet of $\calC$. (Actually, one should take $\pm 1/2$ to get a primitive integral class with our convention, but we ignore this point.)

We have written a Sage program that classifies all the states $s$ for which the hypothesis of Theorem \ref{ad:teo} are satisfied, considered up to symmetries of $\calC$ and up to switching all the signs on facets with the same colours (both these operations produce the same maps $f_s$ up to an isometry of $X$). Among them, we find 63 states that are particularly interesting because the ascending and descending links all collapse to circles: they form a Hopf link at each vertex of the dual cubulation $C$. For such states $s$, Theorem \ref{ad:teo} furnishes a perfect circle-valued Morse function $f_s\colon M \to S^1$ that has precisely one critical point inside each 24-cell of the tessellation. This is coherent, since the total number of critical points is equal to $\chi(M)=8$, which is in turn equal to the number of 24-cells in the tessellation, since $\chi(\calC)=1$.

These 63 states $s$ give rise to potentially topologically different types of perfect circle-valued Morse functions $f_s$ on the same hyperbolic manifold $X$. By construction, in all the cases the function $f_s$ has two singular values $0$ and $1/2$ in $\matR/_\matZ=S^1$, whose counterimages consist of 4 singular points each, like in Figure \ref{functions:fig}. There are two types of regular fibers $f^{-1}(1/4)$ and $f^{-1}(3/4)$ and two types of singular fibers $f^{-1}(0)$ and $f^{-1}(1/2)$. Each regular fiber $X^{\rm reg}$ is a 3-manifold with 24 torus boundary components, one inside each cusp of $X$. Each singular fiber is a 3-manifold $X^{\rm sing}$ with 28 toric boundary components, four of which have been coned. The two singular fibers are actually homeomorphic, because they are related by the isometry of $X$ that sends $\calC^{(v_1,v_2,v_3)}$ to $\calC^{(1-v_1,1-v_2,1-v_3)}$ identically.

\begin{figure}
 \begin{center}
  \labellist
\small\hair 2pt
\pinlabel $f$ at 110 40
\pinlabel $X$ at 10 220
\pinlabel $X^{\rm reg}$ at 15 130
\pinlabel $X^{\rm sing}$ at 78 82
\pinlabel $X^{\rm reg}$ at 125 200
\pinlabel $X^{\rm sing}$ at 147 65
\pinlabel $X^{\rm reg}$ at 185 130
\pinlabel $0$ at 75 20 
\pinlabel $1/2$ at 140 20 
\pinlabel $\tilde f$ at 417 40
\pinlabel $\widetilde X$ at 310 220
\pinlabel $X^{\rm reg}$ at 290 193
\pinlabel $X^{\rm sing}$ at 260 65
\pinlabel $X^{\rm reg}$ at 373 200
\pinlabel $X^{\rm sing}$ at 328 82
\pinlabel $X^{\rm reg}$ at 420 193
\pinlabel $X^{\rm sing}$ at 395 65
\pinlabel $X^{\rm reg}$ at 503 200
\pinlabel $X^{\rm sing}$ at 458 82
\pinlabel $X^{\rm reg}$ at 550 193
\pinlabel $X^{\rm sing}$ at 525 65
\pinlabel $-1/2$ at 260 20 
\pinlabel $0$ at 330 20 
\pinlabel $1/2$ at 395 20 
\pinlabel $1$ at 460 20 
\pinlabel $3/2$ at 525 20 
\endlabellist
  \includegraphics[width = 12.5 cm]{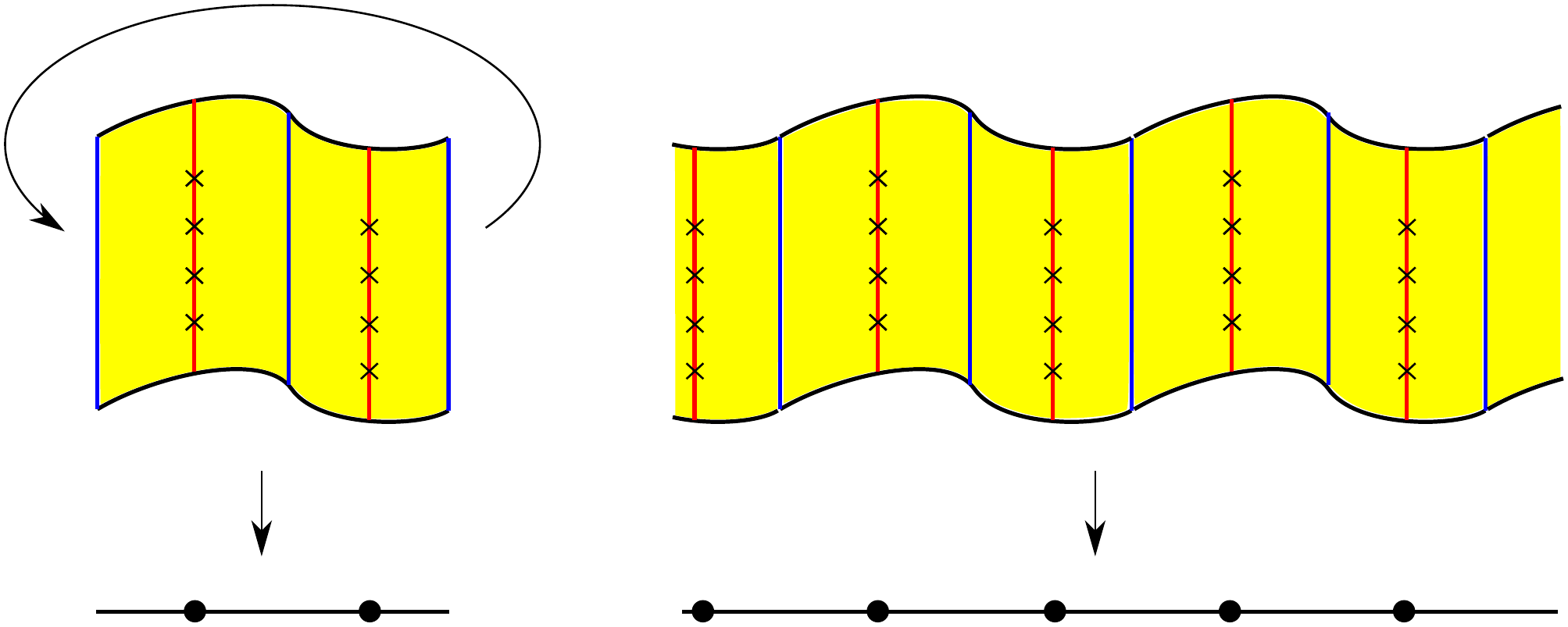}
 \end{center}
 \nota{The circle-valued Morse function $f\colon X \to S^1$ (left) and its lift $\tilde f \colon \widetilde X \to \matR$ (right). There are two critical values 0 and $1/2$ for $f$, each containing 4 critical points in its counterimage, indicated with an X.}
 \label{functions:fig}
\end{figure}

We have determined $X^{\rm sing}$ for each state $s$ using a Sage code \cite{codeM} that uses both Regina \cite{BBP} and SnapPy \cite{Sna}. The manifold $X^{\rm sing}$ has 28 cusps and its first homology is $\matZ^{28}$ in all cases. The manifolds are listed in Tables \ref{fibers:table} and \ref{fibers2:table}. They are all hyperbolic and, quite remarkably, they can all be distinguished by their volumes. Without the great help of hyperbolic geometry it would have been very difficult to distinguish these 3-manifolds. Each manifold has an Epstein -- Penner canonical decomposition that is further cut into geometric hyperbolic ideal tetrahedra, whose number is shown in the last column. The data in the tables should be interpreted as numerical results and not as rigorous proofs: in case of need, one may wish to use the SnapPy verified computations to try to upgrade this discussion to a rigorous argument, but given the numbers of manifolds and tetrahedra involved we have refrained from doing this.

Finally, each state $s$ gives rise to an abelian cover $\widetilde X_s$. Using the method and code mentioned in Subsection \ref{abeliancover:subsection} and described in \cite{Bat} we could get strong numerical evidence for the infinitesimal deformation of the holonomy representation of $\widetilde X_s$ in all the 63 cases.

%For every state $s$ the holonomy of $\widetilde C_{[0,2]}$ turns out to be infinitesimally rigid. However, quite interestingly, the space of infinitesimal deformations of $\widetilde C_{[0,1]}$ is often non-trivial: its dimension is reported in the last column of Tables \ref{fibers:table} and \ref{fibers2:table}.

\begin{table}
\begin{center}
\begin{tabular}{cccc}
N & Volume & Symmetry group & Tetrahedra  \\
\hline
1 & 194.868788430654 &   nonabelian group of order 192 & 192 \\ 
2 & 189.47275083534 &   D4 & 192 \\ 
3 & 186.874353850537 &   Z/2 + D4 & 192 \\ 
4 & 186.340011593947 &   Z/2 + Z/2 & 192 \\ 
5 & 185.307321263554 &   Z/2 & 192 \\ 
6 & 185.035244912975 &   D4 & 200 \\ 
7 & 184.813075164402 &   Z/2 + D4 & 192 \\ 
8 & 184.301433361768 &   nonabelian group of order 64 & 192 \\ 
9 & 184.066725592654 &   Z/2 + Z/2 & 196 \\ 
10 & 183.87312700617 &   Z/2 & 194 \\ 
11 & 183.866948358501 &   Z/2 & 193 \\ 
12 & 183.544395602474 &   D4 & 192 \\ 
13 & 183.436816803216 &   Z/2 + Z/2 + Z/2 & 192 \\ 
14 & 183.392747517971 &   Z/2 & 194 \\ 
15 & 183.121509768255 &   Z/2 + Z/2 & 196 \\ 
16 & 182.360395141194 &   Z/2 & 192 \\ 
17 & 182.280935940832 &   Z/2 + Z/2 & 192 \\ 
18 & 182.171456556108 &   0 & 203 \\ 
19 & 181.283359592032 &   D4 & 192 \\ 
20 & 181.127484303344 &   Z/2 & 196 \\ 
21 & 181.024645445741 &   Z/2 & 196 \\ 
22 & 180.934130108819 &   Z/2 + Z/2 & 212 \\ 
23 & 180.824987315671 &   D4 & 200 \\ 
24 & 180.660993296517 &   Z/2 & 194 \\ 
25 & 180.450694474362 &   D4 & 200 \\ 
26 & 180.387179283461 &   Z/2 & 201 \\ 
27 & 180.33108563585 &   Z/2 & 190 \\ 
28 & 180.248483292705 &   Z/2 + Z/2 & 196 \\ 
29 & 180.127608138661 &   Z/2 + Z/2 & 198 \\ 
30 & 179.869062465521 &   Z/4 & 200 \\ 
31 & 179.754141009737 &   0 & 197 \\ 
32 & 179.656944120778 &   Z/2 & 196 \\ 
33 & 179.181472240992 &   Z/2 & 194 \\ 
34 & 178.902836506372 &   Z/2 + D4 & 196 \\ 
35 & 178.795804830626 &   Z/2 + D4 & 192 \\ 
36 & 178.709806437094 &   0 & 192 \\ 
37 & 178.550276957958 &   Z/2 + Z/2 & 202 
\end{tabular}
\vspace{.2 cm}
\nota{The singular fibers that we obtained for $X$ by assigning different states to $\calC$. All these hyperbolic 3-manifolds have 28 cusps and $H_1=\matZ^{28}$. The last column shows the number of tetrahedra in the Epstein-Penner canonical decomposition (possibly after some subdivision).}
\label{fibers:table}
\end{center}
\end{table}

\begin{table}
\begin{center}
\begin{tabular}{cccc}
N & Volume & Symmetry group & Tetrahedra \\
\hline
38 & 178.49844092298 &   Z/2 + Z/2 + Z/2 & 196 \\ 
39 & 178.355491610596 &   nonabelian group of order 64 & 192 \\ 
40 & 178.321815160374 &   D4 & 204 \\ 
41 & 177.898810518335 &   D3 & 201 \\ 
42 & 177.794415210729 &   0 & 194 \\ 
43 & 177.551934214513 &   Z/2 + D4 & 188 \\ 
44 & 177.362796117873 &   Z/2 & 198 \\ 
45 & 177.250459763042 &   nonabelian group of order 32 & 200 \\ 
46 & 177.110609049916 &   Z/2 + Z/2 & 196 \\ 
47 & 176.982387058175 &   Z/2 + Z/2 & 198 \\ 
48 & 176.898729129159 &   0 & 198 \\ 
49 & 175.421613694494 &   0 & 198 \\ 
50 & 175.170011870291 &   Z/2 + Z/2 & 192 \\ 
51 & 175.085280565074 &   Z/2 + Z/2 & 188 \\ 
52 & 174.081891361605 &   D4 & 200 \\ 
53 & 173.80796075894 &   Z/2 & 194 \\ 
54 & 173.331415822106 &   Z/2 + Z/2 + Z/4 & 192 \\ 
55 & 173.211174252127 &   D4 & 200 \\ 
56 & 172.692650600288 &   Z/2 & 192 \\ 
57 & 172.581866456537 &   D4 & 184 \\ 
58 & 172.160507136111 &   nonabelian group of order 32 & 192 \\ 
59 & 171.484300906217 &   Z/2 + D4 & 184 \\ 
60 & 170.918046348271 &   Z/2 + Z/4 & 192 \\ 
61 & 166.465531880726 &   Z/2 + Z/2 & 186 \\ 
62 & 163.949780648344 &   Z/4 & 184 \\ 
63 & 154.990534348097 &   Z/2 + Z/2 + D8 & 176  
\end{tabular}
\vspace{.2 cm}
\nota{The singular fibers that we obtained for $X$ by assigning different states to $\calC$. All these hyperbolic 3-manifolds have 28 cusps and $H_1=\matZ^{28}$. The last column shows the number of tetrahedra in the Epstein-Penner canonical decomposition (possibly after some subdivision). (continue).}
\label{fibers2:table}
\end{center}
\end{table}

\subsubsection{A very symmetric case} \label{very:symmetric:subsubsection}
The largest hyperbolic manifold in Table \ref{fibers:table} deserves more attention. By construction, every manifold $X^{\rm sing}$ among those listed in the tables has a topological triangulation with 192 tetrahedra, and hence volume $\leq 192 V$ where $V$ is the volume of the ideal regular tetrahedron (the canonical decomposition may contain less or more than 192 tetrahedra). The first manifold in Table \ref{fibers:table} decomposes precisely into 192 regular ideal tetrahedra, so it has the largest volume, and moreover also the largest symmetry group, yet of order 192. The manifold arises from a very symmetric state $s$ that we now describe.

\begin{figure}
 \begin{center}
  \includegraphics[width = 12.5 cm]{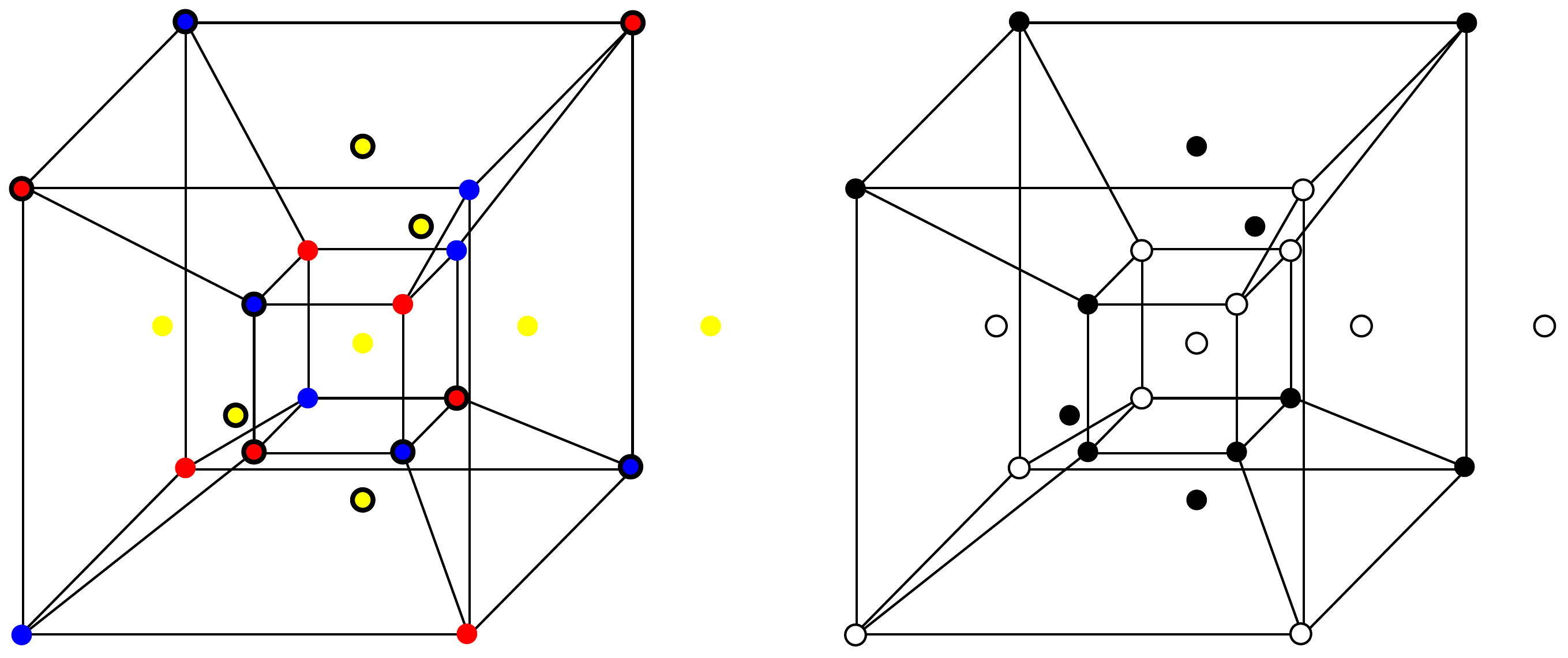}
 \end{center}
 \nota{The very symmetric state $s$. The black and white vertices of $\calC^*$ represent the status I and O respectively (right). In the left figure we show both the colouring and the state.}
 \label{colori:fig}
\end{figure}

This very symmetric state $s$ is shown in Figure \ref{colori:fig}. It has the following appealing algebraic description: each of the three lateral classes of $Q_8$ inside $T_{24}^*$ is preserved by the left multiplication by $i$, and the action decomposes it into two orbits of four vertices each. Assign the status $+1$ (also called O) to one orbit and $=-1$ (also called I) to the other. Up to symmetries, the resulting state $s$ is independent of the choice of the orbits; moreover, we would get the same $s$ (up to symmetries) also if we chose either left- or right-multiplication of any of $\pm i, \pm j, \pm k$.

\begin{figure}
 \begin{center}
  \includegraphics[width = 7 cm]{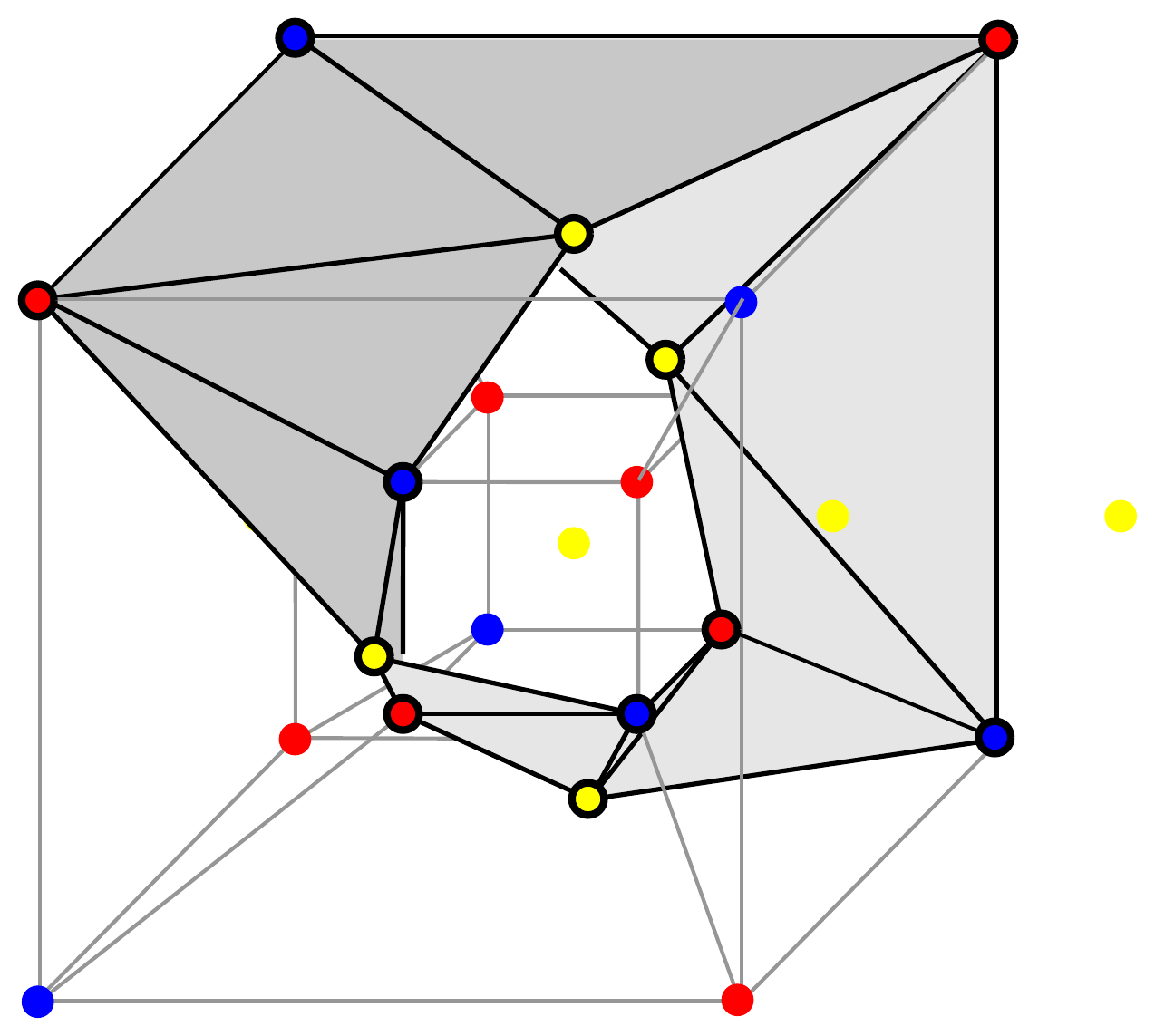}
 \end{center}
 \nota{The descending link is a band decomposed into 12 triangles. The ascending link is isomorphic to it, and altogether they form two bands that collapse onto a Hopf link in $S^3$.}
 \label{nastri:fig}
\end{figure}

In this case, if we apply the algorithm of Remark \ref{algorithm:rem} to detect the ascending and descending links at every vertex $v\in \matZ_2^3$, we find the same state $s$ (up to symmetries of $\calC$) at every interior vertex $v$. At every $v$ the ascending and descending links form altogether a pair of bands in $S^3$ that collapse to a Hopf link. One band is pictured in Figure \ref{nastri:fig}.

With some patience, we can prove by hand that in this very symmetric context the singular fiber $X^{\rm sing}$ decomposes into 192 ideal tetrahedra, and that every edge of the resulting triangulation is adjacent to precisely 6 of them. This implies immediately that $X^{\rm sing}$ has a hyperbolic structure obtained by assigning to each tetrahedron the structure of a regular ideal one.

\begin{figure}
 \begin{center}
  \makebox[\textwidth][c]{\includegraphics[width = 16 cm]{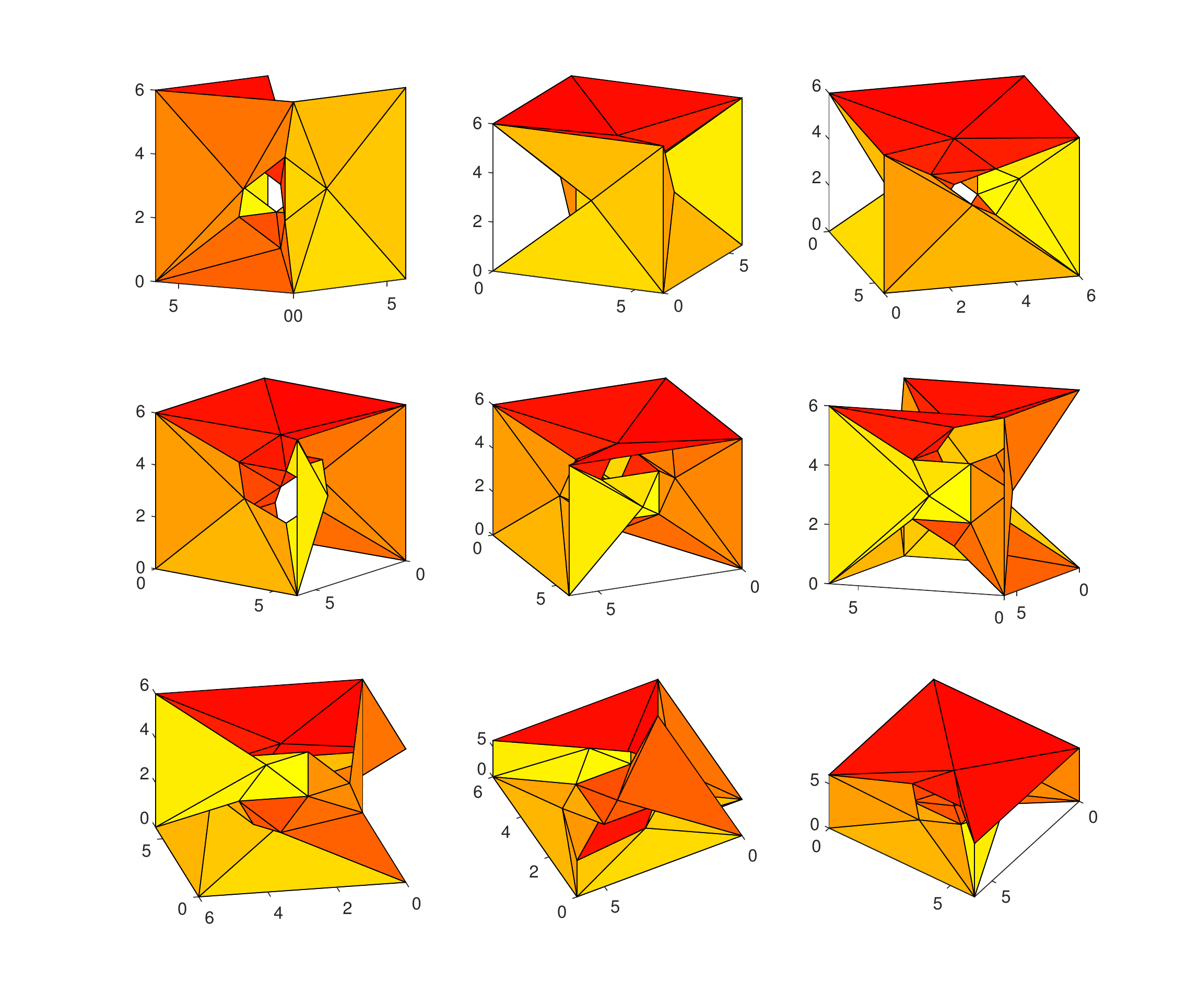}}
 \end{center}
 \nota{The triangles in the 2-skeleton of the 24-cell that separate facets with opposite I/O status form a torus $T$, shown here from nine different viewpoints. In the pictures, we omit for the sake of clarity six triangles that are outside of the central cube and have a vertex at infinity. We may see that every vertex is adjacent to six triangles.}
 \label{torus:fig}
\end{figure}

The ideal triangulation of $X^{\rm sing}$ can be constructed by taking, for each $v\in \matZ_2^3$, the cone with vertex the center of $\calC$ over all the ideal triangles in $\calC$ that separate two facets having opposite status. These ideal triangles form a torus (with 24 punctures) shown in Figure \ref{torus:fig}. One can check that every vertex of the triangulated torus is adjacent to exactly six triangles. 

In this very symmetric case we were able to calculate also the regular fiber $X^{\rm reg}$, that is a hyperbolic 3-manifold with 24 cusps and volume $\sim 152.510077$. It may be identified as the double cover of the (quite messy) orbifold $O$ shown in Figure \ref{orbifold:fig}. The orbifold $O$ is tessellated into 24 (non regular!) ideal octahedra, precisely as the boundary of the 24-cell, and $X^{\rm reg}$ is tessellated into 48 of them.

\begin{figure}
 \centering
  \makebox[\textwidth][c]{\includegraphics[width=16 cm, height= 3.5 cm]{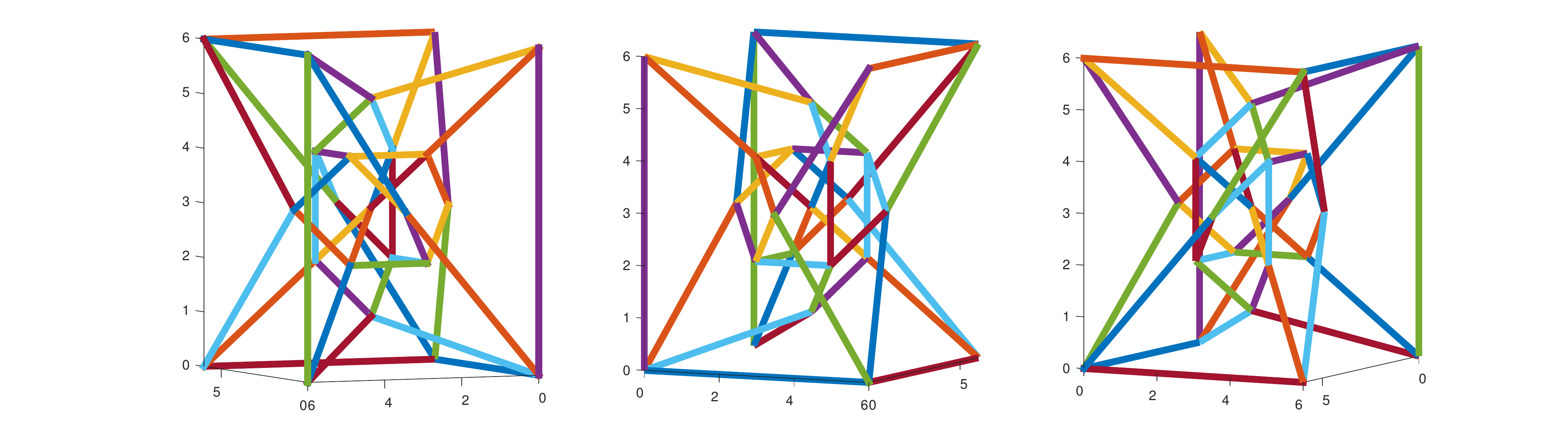}}
 \nota{The hyperbolic orbifold $O$ has base space $S^3$ (with 24 punctures) and singular locus the (quite complicated) knotted 4-valent graph shown here with respect to three different viewpoints. The graph has 24 vertices that should be removed and indicate 24 cusps, each based on the flat orbifold $(S^2,2,2,2,2)$. For the sake of clarity, we omit the vertex at infinity and four edges that are connected to it. The graph is contained in the 1-skeleton of the 24-cell. It has many symmetries that are unfortunately not apparent from the figure.}
 \label{orbifold:fig}
\end{figure}

\begin{rem}[Fibers are pleated]
We should mention that, like in the more familiar fibrations in dimension 3, the fibers $X^{\rm reg}$ and $X^{\rm sing}$ are embedded in $X$ in a very pleated way: they are decomposed respectively into octahedra and tetrahedra that form some angles along their 2-dimensional faces, in such a complicated way that any connected component $\bar X$ of the preimage of $X^{\rm reg}$ or $X^{\rm sing}$ in $\matH^4$ has the whole sphere at infinity $S^3 = \partial \matH^4$ as a limit set. This is a general fact when we analyse perfect circle-valued Morse functions on hyperbolic 4-manifolds.

Remember that the fibers $X^{\rm reg}$ and $X^{\rm sing}$ are not $\pi_1$-injective in $X$, because the index-2 critical points annihilate some of the elements in their fundamental groups. This implies that the connected component $\bar X$ just mentioned is not simply connected.
\end{rem}

Finally, for this very symmetric state $s$ we could prove rigorously the following fact, see \cite{Bat} for details.

\begin{teo} \label{rigid:teo}
The holonomy representation of the abelian cover $\widetilde X$ is infinitesimally, and hence locally, rigid.
\end{teo} 

\subsection{A very small orbifold example} \label{Y:subsection}
Let $\widetilde X$ be the abelian cover of $X$ constructed from the very symmetric state of Section \ref{very:symmetric:subsubsection}. The manifolds $X$ and $\widetilde X$ have plenty of symmetries, many of which preserve the fibers of $f$. It is therefore tempting to try to quotient them by some symmetries, in order to find smaller examples. 

By quotienting $\widetilde X$ with an appropriate group of symmetries we have found a very small hyperbolic 4-orbifold $Y$ with a circle-valued Morse function $f\colon Y \to S^1$ whose singular and regular fibers appear in the very first segment of the SnapPea census \cite{CHW} as the manifolds ${\tt m036}$ and ${\tt m203}$. 

The construction of $Y$ may be of independent interest, so we briefly describe it. Let $\calC$ be the ideal regular 24-cell, with center $v$. Let $P$ be the cone on $v$ over a octahedral facet $O$ of $\calC$. Then $P$ is a 4-dimensional pyramid, with a octahedral base and eight tetrahedral lateral facets. It has 8 ideal vertices and one real one. Its dihedral angles are $2\pi/3$ at the triangles containing $v$ (the lateral ones), and $\pi/4$ at the triangles contained in $O$ (the base ones). It is not a Coxeter polytope since $2\pi/3$ does not divide $\pi$, but it may be used nevertheless to construct manifolds.

\begin{figure}
 \begin{center}
  \labellist
\small\hair 2pt
\pinlabel $s$ at 160 20
\endlabellist
  \includegraphics[width = 10 cm]{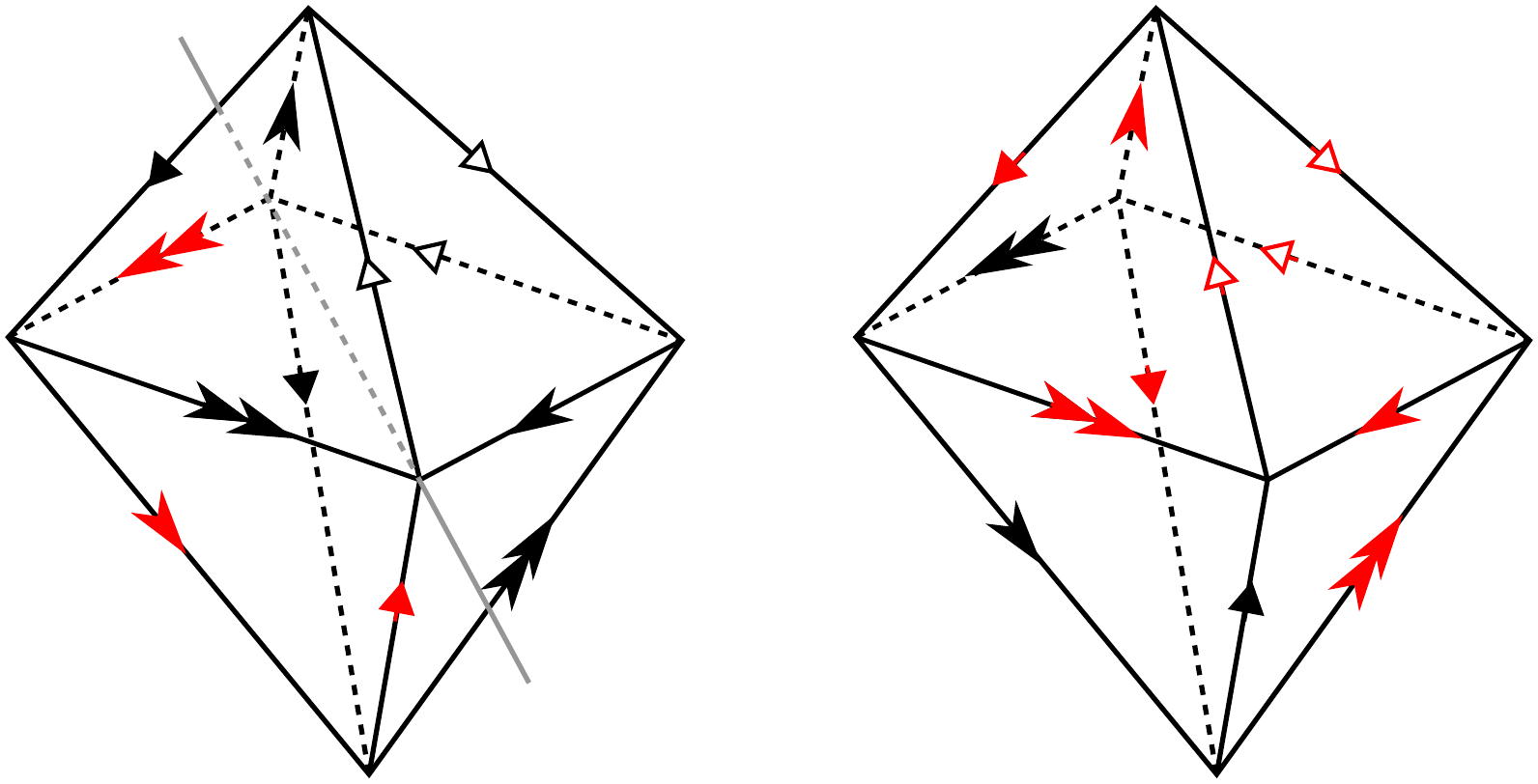}
 \end{center}
 \nota{There is only one way to pair the faces of the two octahedra so that all the symbols at the edges match. The result is the lens space $L(12,5)$. }
 \label{new_lens:fig}
\end{figure}

Figure \ref{new_lens:fig} shows two octahedra $O_1$ and $O_2$ and a face pairing between them giving the lens space $L(12,5)$. Note that every edge has valence 3: if we give each $O_i$ the structure of a spherical regular octahedron with dihedral angles $2\pi/3$, we get $L(12,5)$ with its spherical metric. 

We construct $Y$ by picking two copies $P_1$ and $P_2$ of $P$, and identifying their lateral facets by extending the pairing shown in Figure \ref{new_lens:fig} for their base octahedra $O_1$ and $O_2$. Then we glue $O_1$ to $O_2$ along the following isometry: first rotate $O_1$ of an angle $\pi$ along the axis $s$ shown in Figure \ref{new_lens:fig}, and then translate it to $O_2$. To verify that we get an orbifold $Y$ with $v$ as the only singular point, we just check that this identification produces two orbits of eight triangles, and since $8 \cdot \pi/4 = 2\pi$ we are done.

The orbifold $Y$ has a single singular point with link $L(12,5)$, a single cusp, and $\chi(Y) = 1/12$. It can in fact be obtained by quotienting $\widetilde X$ by an appropriate group of isometries, and it inherits from it a circle-valued ``Morse function'' $f\colon Y \to S^1$, that has a single ``index-2 singular point'' at $v$. (A Morse function $f$ on a 4-manifold near an index-2 critical point is equivalent to $f(x_1,x_2,x_3,x_4) = x_1^2+x_2^2 -x_3^2-x_4^2$. Such a function $f$ is $\SO(2) \times \SO(2)$-invariant and hence descends to a function on the cone over any lens space.)

\begin{figure}
 \begin{center}
  \includegraphics[width = 3 cm]{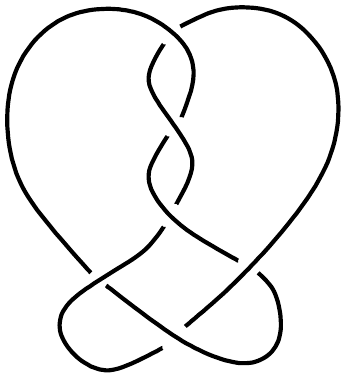}
 \end{center}
 \nota{The link {\tt L6a2} from Thistelthwaite's link table. Its complement is the hyperbolic manifold ${\tt m 203}$.}
 \label{L6a2:intro:fig}
\end{figure}

\begin{figure}
 \begin{center}
  \includegraphics[width = 4.5 cm]{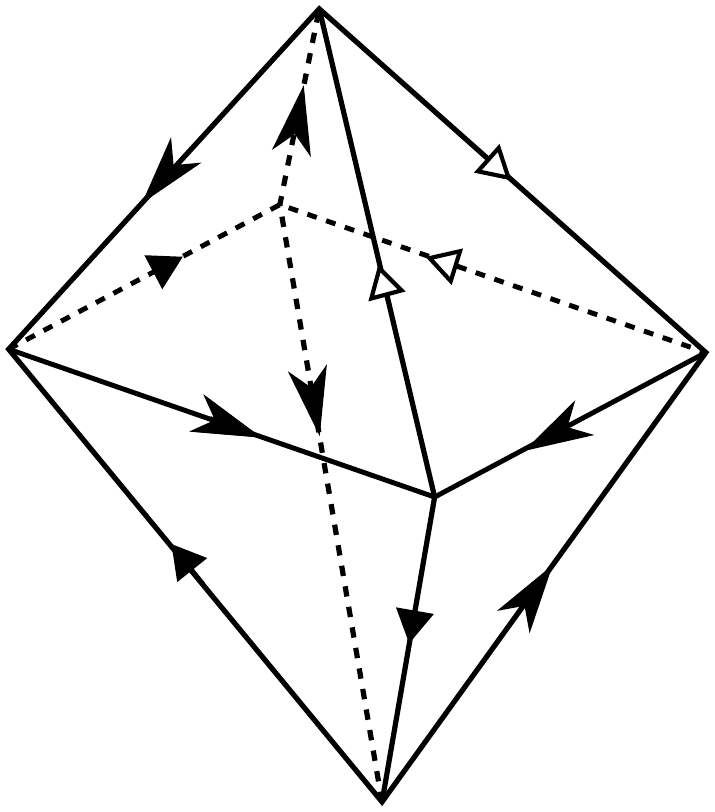}
 \end{center}
 \nota{The hyperbolic manifold ${\tt m036}$ is obtained by pairing the faces of the octahedron with the unique pair of isometries that match the arrows on the edges. The manifold so obtained has three edges with valence 3, 3, and 6.}
 \label{m036:fig}
\end{figure}

The regular fiber of $f$ is the single-cusped manifold $Y^{\rm reg} = {\tt m036}$ from the census \cite{CHW}, with volume $\sim$ 3.177293 and homology $\matZ \times \matZ_3$. It decomposes as a single (non regular) ideal octahedron \cite{HPP} as in Figure \ref{m036:fig}. The singular fiber is the twice-cusped $Y^{\rm sing} ={\tt m203}$, that is the complement of the link in Figure \ref{L6a2:intro:fig}, with one cusp coned; it has volume $\sim$ 4.059766 and decomposes into four regular ideal tetrahedra \cite{CHW}. The manifolds $Y^{\rm reg}$ and $Y^{\rm sing}$ are covered by $X^{\rm reg}$ and $X^{\rm sing}$.

Using SnapPy we discover that if we Dehn fill one cusp of {\tt m203} by filling the slope $\pm \frac 32$ we indeed get {\tt m036}. It is not necessary to specify which component of the link is filled since they are symmetric. SnapPy also says that there is an isometry $\psi$ of ${\tt m203}$ that sends the slope $\frac rs$ to $-\frac rs$, and this explains why we get the same manifold ${\tt m036}$ by filling the slope $\pm \frac 32$. The intersection between the two slopes is $\det \matr 3{-3}22 = 12$, coherently with the fact that $v$ is a cone over $L(12,5)$. As with the manifold $W$ considered in Section \ref{W:subsection}, using SnapPy we also discover that ${\tt m036}$ has an \emph{additional isometry} $\varphi$ that is not induced from ${\tt m203}$. A ping-pong dynamics between $\varphi$ and $\psi$ as in Remark \ref{ping:pong:rem} is very likely to hold also here.

\subsubsection{A geometrically infinite Kleinian group in $\matH^4$ with 2 generators}
The detour from $X$ to $Y$ also led us to find an explicit holonomy $\Phi$ for the regular fiber $Y^{\rm reg}$ of $Y$. The holonomy is interesting because its image in $\Iso(\matH^4)$ is discrete and has limit set $S^3$ and the group has only two generators. However, recall that the holonomy is not injective because $Y^{\rm reg}$ is not $\pi_1$-injective in $Y$.

The holonomy is the following. A presentation for $\pi_1({\tt m036})$ is
$$\langle\ a, b\ |\ a^{-2}b^{-2}ab^{-2}a^{-2}b \ \rangle.$$
We have  
$$\Phi(a) = 
\begin{pmatrix}
1/2 & -1/2 & 1/2 & 1/2  & 0 \\
1/2 & 1/2 & -1/2 & 1/2 & 0 \\
1/2 & -1/2 & -1/2 & -1/2 & 0 \\
1/2 & 1/2 & 1/2 & -1/2 & 0 \\
0 & 0 & 0 & 0 & 1
\end{pmatrix}$$
and
$$
\Phi(b) =
\begin{pmatrix}
-7/2 & 3/2 & 3/2 & 3/2 & 3\sqrt 2 \\
 -3/2   &     1/2  &      1/2  &     -1/2   & \sqrt 2 \\
       3/2   &     1/2   &    -1/2  &     -1/2 &  -\sqrt 2 \\
       3/2    &   -1/2   &     1/2   &    -1/2  & -\sqrt 2 \\
-3\sqrt 2  &  \sqrt 2 &   \sqrt 2  &  \sqrt 2  &        5 
\end{pmatrix}.
$$
Both $\Phi(a)$ and $\Phi(b)$ are elliptic of order 12, while $a$ and $b$ have infinite order in $\pi_1({\tt m036})$. The representation $\Phi$ is \emph{type-preserving}, that is it sends parabolics to parabolics. The geometrically infinite discrete group $\Gamma = \Phi(\pi_1({\tt m036})) $ has only two generators and these are expressed explicitly above as some elliptic transformations. The limit set of $\Gamma$ is the whole sphere at infinity $S^3 = \partial \matH^4$. The group $\Gamma$ has infinitely many conjugacy classes of finite subgroups: the existence of such phenomena in dimension 4 is not a surprise \cite{Kf}. 

We can also descend one dimension further: we discover from \cite{Bell} that ${\tt m036}$ fibers over $S^1$ with fiber the surface $\Sigma$ with genus 2 and one puncture. Therefore there is a type-preserving representation $\Phi\colon \pi_1(\Sigma) \to \Isom(\matH^4)$ whose image is a discrete geometrically infinite group with limit set $S^3 = \partial \matH^4$.

\begin{quest} Can we deform type-preservingly the representations $\Phi$ of $\pi_1({\tt m036})$ and $\pi_1(\Sigma)$? Can we find a path connecting them to their Fuchsian representations, or to any geometrically finite discrete representation? 
\end{quest}

\subsection{A compact example} \label{Z:subsection}
We finish by describing the compact example $Z$. Let $P$ be the compact right-angled 120-cell. It has a natural 5-colouring that may be described using quaternions as follows. The facets of $P$ are naturally identified with 120 elements of the binary icosahedral group $I^*_{120}$. This contains the binary tetrahedral group $T^*_{24}$ as an index-5 (not normal) subgroup. The five left lateral classes of $T^*_{24}$ in $I^*_{120}$ furnish a 5-colouring for $P$.

Recall that $\chi(P) = 17/2$. The 5-colouring produces a compact hyperbolic 4-manifold $Z$ with $\chi(Z) = 272$. To build the state $s$, we try to mimic the algebraic construction that worked very well with the 24-cell: we pick the state for $T_{24}^*$ described in Section \ref{very:symmetric:subsubsection} and we extend it to the lateral classes by left-multiplying with some given elements. Using Sage we find that one $s$ constructed in this way fulfills the requirements of Theorem \ref{ad:teo} and hence produces a perfect circle-valued Morse function $f\colon Z \to S^1$ with as much as 272 singular points of index 2.

An algebraic fibration on a manifold tessellated by copies of the 120-cell was already constructed in \cite{JNW}.

\begin{rem}[Large Euler characteristic]
With $P_4$ and $\calC$ the ascending and descending links collapsed either to a point or to a circle. This allowed us to conclude that each copy of $P_4$ and $\calC$ in the decomposition contains either 0 or 1 singular points. This was possible since $\chi(P_4) = 1/16$ and $\chi(\calC)=1$, so the average number of singular points in each polytope of the decomposition is $1/16$ and $1$ in these cases. 

Here we have $\chi(P) = 17/2$, and this is the average number of singular points in each copy of $P$. It is therefore impossible to find only ascending and descending links that collapse to points or circles: some of them must collapse to more complicated 1-complexes. As a consequence, the Morse function $f_s$ is not canonically determined by the combinatorics: we are coning along some Heegaard graphs in $S^3$ and there are many ways to perturb this to a disjoint simultaneous attachment of 2-handles. For this reason we did not make any attempt to determine the singular fibers.
\end{rem}

\section{Related results and open questions} \label{questions:section}
The main inspiration of this work is the paper \cite{JNW} where the authors construct some algebraic fibrations on some hyperbolic 4-manifolds that cover the right-angled 24- and 120-cells. Our first contribution here is to promote these algebraic fibrations to perfect circle-valued Morse functions.
Other related recent contributions to finding algebraic fibrations are \cite{AS, FV, Kie}.

Like fibrations, perfect circle-valued Morse functions lift to any finite-sheeted cover. It therefore makes sense to ask the following rather far-reaching question.

\begin{quest}
Does every finite-volume hyperbolic $n$-manifold have a finite cover that has a perfect circle-valued Morse function?
\end{quest}

For the moment not a single example of perfect circle-valued Morse function is known in dimension $n\geq 5$. Given the lifting property of such functions, the question can be rephrased as follows: does every commensurability class contain some representative that has a perfect circle-valued Morse function?

Concerning dimension $n=4$, the examples exhibited in this paper show that the answer is positive for the commensurability classes of $P_4$, the right-angled 24-cell, and the right-angled 120-cell. The first two commensurability classes are actually the same \cite{RT4}. Every Coxeter simplex in dimension four belongs to one of these two classes \cite{JKRT}, see \cite{M} for a survey on hyperbolic 4-manifolds. So in particular the Davis manifold \cite{Da} and the Conder -- Maclachlan manifold \cite{CM} also belong to the second commensurability class, and in particular they virtually have a perfect circle-valued Morse function.

It is easy to construct a hyperbolic 4-manifold that does \emph{not} admit a perfect circle-valued Morse function: it suffices to pick a cusped one that has at least one cusp section homeomorphic to the Hantzsche-Wendt manifold. Since this flat 3-manifold does not fiber, the cusped 4-manifold has no circle-valued Morse function of any kind. All the 22 orientable manifolds in \cite{RT4} contain such a section. There is also a cusped hyperbolic 4-manifold whose cusp sections are all of this kind \cite{FKS}.

Another obvious obstruction to having a perfect circle-valued Morse function is the vanishing of the first Betti number over the real numbers. Some of the cusped manifolds in \cite{RT4} have $b_1=0$, but no compact hyperbolic 4-manifold with $b_1=0$ seems known at present. In particular, we do not know a single example of a compact hyperbolic 4-manifold that is known not to have a perfect circle-valued Morse function.

We have seen in Theorem \ref{W:teo} that the open set $U$ is \emph{polytopal}, that is it is the cone over the open facets of a polytope.

\begin{quest} Let $M$ be a hyperbolic $n$-manifold with $n\geq 3$. Is there always an open polytopal subset $U\subset H^1(M;\matR)$ such that the classes represented by circle-valued Morse functions form the intersection $U\cap H^1(M;\matZ)$?
\end{quest}

This could be related to the fact that the BNS invariant is of polytopal type for many groups, see \cite{K}. The cohomology classes represented by perfect circle-valued Morse functions are contained in the BNS invariant of $\pi_1(M)$ since they have finitely generated kernel.

\begin{quest} Let $M^{\rm reg}$ be a regular fiber of a perfect circle-valued Morse function on a finite-volume hyperbolic 4-manifold $M$. Is $M^{\rm reg}$ necessarily hyperbolic? Which hyperbolic 3-manifolds can arise in this way?
\end{quest}

If $M^{\rm reg}$ were not aspherical, we would get a counterexample to Whitehead's asphericity conjecture, since the abelian cover of $M$ is aspherical and obtained from $M^{\rm reg}$ by attaching infinitely many 2-handles.


\begin{thebibliography}{99}

\bibitem{A} \textsc{I.~Agol}, \emph{The virtual Haken conjecture} (with an appendix by I. Agol, D. Groves and J.
Manning), Doc. Math. \textbf{18} (2013), 1045--1087.

\bibitem{AS} \textsc{I.~Agol -- M.~Stover}, \emph{Congruence RFRS towers}, With an appendix by M. Seng\" un {\tt arXiv:1912.10283}, to appear in Ann. Inst. Fourier.

\bibitem{ALR} \textsc{I. Agol -- D. Long -- A. Reid}, 
\emph{The Bianchi groups are separable on geometrically finite subgroups}, Ann. Math., \textbf{153} (2001), 599--621.

\bibitem{Bat} \textsc{L.~Battista}, \emph{Infinitesimal Rigidity of Cubulated Manifolds}, in preparation.

\bibitem{Bell} \textsc{M.~Bell}, \emph{Recognising Mapping Classes}, PhD thesis.

\bibitem{BB} \textsc{M.~Bestvina -- N.~Brady}, \emph{Morse theory and finiteness properties of groups}, Invent. Math., \textbf{129} (1997), 445--470.

\bibitem{BHW} \textsc{S.~Bleiler -- C.~Hodgson -- J.~Weeks}, \emph{Cosmetic surgery on knots}, in ``Proceedings of the Kirbyfest'' (Berkeley, CA, 1998), 23--34 (electronic), Geometry and Topology Monographs, \textbf{2}, Coventry, 19.

\bibitem{BrBr} \textsc{J.~Brock -- K.~Bromberg}, \emph{On the density of geometrically finite Kleinian groups}. Acta Math., \textbf{192} (2004), 33--93.

\bibitem{B} \textsc{K.~Bromberg}, \emph{Projective structures with degenerate holonomy and the Bers density conjecture}, Ann. of Math. \textbf{166} (2007), 77--93.

\bibitem{BBP} \textsc{B. Burton -- R. Budney -- W. Pettersson et al.},
Regina: Software for low-dimensional topology,
http://regina-normal.github.io/, 1999--2021.

\bibitem{CHW} \textsc{P.~Callahan -- M.~Hildebrand -- J.~Weeks}, \emph{A census of cusped hyperbolic 3-manifolds}, Math. Comp. \textbf{68} (1999), 321--332.

\bibitem{CM} \textsc{M.~Conder -- C.~Maclachlan}, \emph{Small volume compact hyperbolic 4-manifolds}, Proc. Amer. Math. Soc. \textbf{133} (2005), 2469--2476.

\bibitem{Sna} \textsc{M.~Culler -- N.~Dunfield -- M. G\"orner  -- J.~Weeks}, \emph{SnapPy, a computer
program for studying the geometry and topology of 3-manifolds}, {\tt 
http://www.math.uic.edu/t3m/SnapPy/}

\bibitem{Da} \textsc{M.~Davis}, \emph{A hyperbolic 4-manifold}, Proc. Amer. Math. Soc. \textbf{93} (1985), 325--328.

\bibitem{ERT} \textsc{B.~Everitt -- J.~Ratcliffe -- S.~Tschantz}, \emph{Right-angled Coxeter polytopes, hyperbolic six-manifolds, and a problem of Siegel}, Math. Ann. \textbf{354} (2012), 871--905.

\bibitem{FKS} \textsc{L~Ferrari -- A.~Kolpakov -- L.~Slavich}, \emph{Cusps of Hyperbolic 4-Manifolds and Rational Homology Spheres}, {\tt arXiv:2009.09995},  to appear on Proc. London Math. Soc.

\bibitem{FV} \textsc{S.~Friedl -- S.~Vidussi}, \emph{Virtual algebraic fibrations of K\"ahler groups}, {\tt arXiv:1704.07041}, to appear in Nagoya Math. J.

\bibitem{HPP} \textsc{D.~Heard -- E.~Pervova -- C.~Petronio}, \emph{The 191 orientable octahedral manifolds}. Experiment. Math., \textbf{17} (2008), 473--486.

\bibitem{HM} \textsc{M. W. Hirsch -- B. Mazur}, ``Smoothings of piecewise linear manifolds,'' Princeton University Press, Princeton, NJ, 1974, Annals of Mathematics Studies, No. 80.

\bibitem{IMM} \textsc{G.~Italiano -- B.~Martelli -- M.~Migliorini},
\emph{Hyperbolic manifolds that fiber algebraically up to dimension 8}, {\tt arXiv:2010.10200}

\bibitem{JNW} \textsc{K.~Jankiewicz -- S.~Norin -- D.~T.~Wise},
\emph{Virtually fibering right-angled Coxeter groups}, J.~Inst.~Math.~Jussieu.
\textbf{20} (2021), 957--987.

\bibitem{JKRT} \textsc{N. Johnson -- R. Kellerhals -- J. Ratcliffe -- S. Tschantz}, 
\emph{The size of a hyperbolic Coxeter simplex},
Transformation Groups \textbf{4} (1999), 329--353.

\bibitem{Kf} \textsc{M.~Kapovich}, \emph{On the absence of Sullivan's cusp finiteness theorem in higher dimensions}, in ``Algebra and analysis'' (Irkutsk, 1989), Amer. Math. Soc., Providence, RI, 1995, pp. 77--89.

\bibitem{Kie} \textsc{D.~Kielak}, \emph{Residually finite rationally-solvable groups and virtual fibring}, J. Amer. Math. Soc. \textbf{33} (2020), 451--486.

\bibitem{K} \bysame, \emph{The Bieri-Neumann-Strebel invariants via Newton polytopes}, Inventiones Math. \textbf{219} (2020), 1009--1068.

\bibitem{KS} \textsc{S. Kerckhoff -- P. Storm}, \emph{Local rigidity of hyperbolic manifolds with geodesic boundary}. J. Topol.,
\textbf{5} (2012), 757--784.

\bibitem{KM} \textsc{A.~Kolpakov -- B.~Martelli}, \emph{Hyperbolic four-manifolds with one cusp}, Geom. \& Funct. Anal. \textbf{23} (2013), 1903--1933.

\bibitem{KSla} \textsc{A.~Kolpakov -- L.~Slavich},
\emph{Hyperbolic four-manifolds, colourings and mutations}, Proc. London Math. Soc. \textbf{113} (2016), 163--184.

\bibitem{M} \textsc{B.~Martelli}, \emph{Hyperbolic four-manifolds}, Handbook of Group Actions, Volume III. Advanced Lectures in 
Mathematics series \textbf{40} (2018), 37--58.

\bibitem{MR} \textsc{B.~Martelli -- S.~Riolo}, \emph{Hyperbolic Dehn filling in dimension four}, Geom. \& Topol. \textbf{22} (2018), 1647--1716.

\bibitem{Mun} \textsc{J.~R.~Munkres}, \emph{Obstructions to the smoothing of piecewise-differentiable homeomorphisms}, Ann. of Math. (2) \textbf{72} (1960), 521--554.

\bibitem{NS} \textsc{H.~Namazi -- J.~Souto}, \emph{Non-realizability and ending laminations: Proof of the density conjecture}, Acta Math., \textbf{209} (2012), 323--395.

\bibitem{PV} \textsc{L.~Potyagailo -- E-~V.~Vinberg}, \emph{On right-angled reflection groups in hyperbolic spaces}, Comment. Math. Helv. \textbf{80} (2005), 63--73.

\bibitem{RT4} \textsc{J.~Ratcliffe -- S.~Tschantz}, \emph{The volume spectrum of hyperbolic 4-manifolds}, Experiment. Math. \textbf{9} (2000), 101--125.

\bibitem{RS} \textsc{C.~Rourke -- B.~Sanderson},
``Introduction to piecewise-linear topology,'' Springer--Verlag 1972.

\bibitem{Th} \textsc{W.~P.~Thurston}, \emph{A norm for the homology of 3-manifolds}, Mem. Amer. Math. Soc. \textbf{59} (1986), 99--130. 

\bibitem{Wang} \textsc{H.C.~Wang} \emph{Topics on totally discontinuous groups}, Symmetric Spaces, W.M.~Boothby and G.L.~Weiss, eds, Pure Appl. Math. \textbf{8}, Marcel Dekker, New York (1972), 459--487.

\bibitem{W} \textsc{D.~T.~Wise}, \emph{The structure of groups with a quasi-convex hierarchy}, preprint. Available at {\tt http://www.math.mcgill.ca/wise/papers} 

\bibitem{codeM} {\tt http://people.dm.unipi.it/martelli/research.html}

%\bibitem{code} {\tt http://people.dm.unipi.it/battista/code/geom\_inf\_rigid\_4\_manifold/}

\end{thebibliography}
\end{document}